\documentclass[a4paper]{amsart}
\usepackage{graphicx}
\usepackage{amsmath,amssymb,amsfonts}
\usepackage{amsthm}
\usepackage{color}

\newtheorem{thm}{Theorem}[section]
\newtheorem{cor}[thm]{Corollary}
\newtheorem{lem}[thm]{Lemma}
\newtheorem{prop}[thm]{Proposition}

\theoremstyle{definition}
\newtheorem{defn}[thm]{Definition}
\theoremstyle{remark}
\newtheorem{rem}[thm]{Remark}
\newtheorem{exa}[thm]{Example}

\numberwithin{equation}{section}
\newtheorem{ques}[thm]{Question}

\newcommand{\eps}{\varepsilon}

\newcommand{\lam}{\lambda}
\newcommand{\Lam}{\Lambda}

\DeclareMathOperator{\dist}{dist}
\DeclareMathOperator{\Span}{Span}
\DeclareMathOperator{\Orb}{Orb}

\DeclareMathOperator{\Hol}{Hol}

\newcommand{\eq}{\equiv}
\newcommand{\vphi}{\varphi}
\newcommand{\ol}[1]{\overline{#1}}
\newcommand{\wt}[1]{\widetilde{#1}}

\newcommand{\bigbracket}[1]{\left( #1 \right)}

\newcommand{\bigcurlybracket}[1]{\left\{ #1 \right\}}

\newcommand{\bigabs}[1]{\left| #1 \right|}
\newcommand{\bignorm}[1]{\left\| #1 \right\|}

\newcommand{\calC}{\mathcal C}

\newcommand{\calH}{\mathcal H}

\newcommand {\R} {\mathbb R}
\newcommand {\C} {\mathbb C}

\newcommand {\N} {\mathbb N}
\newcommand {\D} {\mathbb D}

\newcommand{\lbsq}{\ell_\beta^2}

\newcommand{\HbS}{\calH(E, \beta_S)}

\newcommand{\Hblam}{\calH(\beta,\Lam)}

\def\dsum{\displaystyle\sum}
\def\dfrac{\displaystyle\frac}

\definecolor{ao(english)}{rgb}{0.0, 0.5, 0.0}

\definecolor{darkmagenta}{rgb}{0.55, 0.0, 0.55}

\begin{document}
	
\title[]{Weighted holomorphic Dirichlet series and composition operators with polynomial symbols}%
	
\date{\today}%
	
\author{Emmanuel Fricain$^{1}$}%
\address{Laboratoire Paul Painlev\'e, UFR de Math\'ematiques, B\^atiment M2, Universit\'e de Lille, 59 655 Villeneuve d'Ascq C\'edex, France}%
\email{emmanuel.fricain@univ-lille.fr}%
	
\author{Camille Mau$^{2}$}%
\address{Division of Mathematical Sciences, School of Physical and Mathematical Sciences, Nanyang Technological University (NTU), 637371 Singapore}%
\email{CAMILLE001@e.ntu.edu.sg}

\thanks{Supported in part by the Labex CEMPI  (ANR-11-LABX-0007-01)$^{1}$ and by CN Yang Scholars Programme$^{2}$}
	
	

\thanks{The authors would like to thank L\^e Hai Khoi for many fruithful conversations and comments concerning a preliminary version of this work. They also thank Sophie Grivaux for helpful discussions concerning the cyclicity part of this work.}
	
	
\begin{abstract}
In this paper, we introduce a general class of weighted spaces of holomorphic Dirichlet series (with real frequencies) analytic in some half-plane and study composition operators on these spaces. In the particular case when the symbol inducing the composition operator is an affine function, we give criteria for boundedness and compactness. We also study the cyclicity property and as a byproduct give a sufficient condition so that the direct sum of the identity plus a weighted forward shift operator on the Hardy space $H^2$ is cyclic.
\end{abstract}	
	
\maketitle

\tableofcontents

\section{\bf Introduction}

\subsection{Dirichlet series}
Let $\Lam = (\lambda_n)_{n\geq 1}$ be a strictly increasing sequence of non-negative real numbers tending to $\infty$.
A \emph{Dirichlet series of type $\Lambda$} (Dirichlet series in short) is a series of the form
\begin{align}\label{DS}
\sum_{n=1}^{\infty} a_n e^{-\lambda_n z}\,,
\end{align}
where $z$ is a complex variable, and the coefficients of the series are given by a sequence $(a_n)_{n\geq 1}$ of complex numbers.

If $\lambda_n=n$, $n\geq 1$, then \eqref{DS} are power series in $\zeta=e^{-z}$. When $\lambda_n=\log{n}$, $n\geq 1$, we get the so-called classical Dirichlet series, which have many applications in analytic number theory (see, e.g., \cite{Apo}). Classical Dirichlet series also relate to several problems in functional analysis (see, e.g., \cite{Hed04} and references therein). We refer the reader to the monograph \cite{HR15} or \cite{Valiron} for more detailed information on Dirichlet series.

The properties of convergence of Dirichlet series depend on two specific quantities that we introduce now. Fix $\Lambda=(\lambda_n)_{n\geq 1}$ a strictly increasing sequence of non-negative real numbers tending to $\infty$, and define
\begin{equation}\label{L}
L = \limsup_{n\to\infty} \frac{\log n}{\lambda_n}.
\end{equation}
Now to each Dirichlet series $f$ of type $\Lambda$ given by \eqref{DS} we may associate the following quantity
\begin{equation}\label{L-D}
D_f= \limsup_{n\to\infty} \frac{\log |a_n|}{\lambda_n}.
\end{equation}
Note that the quantity $D_f$ is specific to each Dirichlet series of type $\Lambda$ with coefficients $(a_n)_{n\geq 1}$. Nevertheless, most of the time there is no confusion and we omit the reference to the associated Dirichlet series and write generally $D$.
	
It is known that if a Dirichlet series converges at some $w\in\C$, it converges for all $z$ with $\Re(z) > \Re(w)$, and more precisely, for every fixed non-negative real number $M$, it converges uniformly in the region $\{\Re(z)\geq \Re(w),\,|z-w|\leq M(\Re(z)-\Re(w))\}$, see \cite[page 5]{Valiron}. Let us denote by $\sigma_c$ the abscissa of convergence of a Dirichlet series, which is defined as
\[
\sigma_c = \inf \left\{r\in \R: \sum_{n=1}^\infty a_n e^{-\lambda_nz}\ \text{converges on}\ \C_r \right\},
\]
where $\C_r := \{z: \Re(z)> r\}$. We also need the following extension of the definition of $\C_r$ for $r=\pm\infty$. By convention $\C_\infty=\emptyset$ and $\C_{-\infty}=\C$.

The abscissa $\sigma_u$ of uniform convergence and $\sigma_a$ of absolute convergence are defined in a similar way. So a Dirichlet series converges (resp. uniformly, absolutely) in the right half-plane $\C_{\sigma_c}$ (resp. $\C_{\sigma_u}$, $\C_{\sigma_a}$) to a holomorphic function.

In the case $L < \infty$ the three abscissae are related by the Valiron formula (see, e.g., \cite{YDT})
\begin{equation}\label{eq:Valiron-formula}
D\leq \sigma_c \leq \sigma_u \leq \sigma_a \leq D+L.
\end{equation}
In particular, the Dirichlet series \eqref{DS} represents an entire function if and only if $D=-\infty$.

\subsection{Composition operators}

Let $X$ be a functional Banach space and suppose that all $f\in X$ have a common domain $G \subseteq\C$. Let $\vphi$ be an analytic self-map on $G$. Then $\vphi$ induces the \emph{composition operator} $C_\vphi$ on $X$ defined by
\begin{align*}
C_\vphi(f) = f\circ \vphi,\qquad \forall f\in X.
\end{align*}

The first natural and non trivial question is to know if $C_\vphi$ maps $X$ into itself, and if this is the case, what can be said about standard spectral properties of $C_\vphi$ as an operator on $X$. There is a rich literature on this topic when $X$ is the Hardy space, the Bergman space or the Dirichlet space (see for instance \cite{CM,Sh}). The situation of spaces of classical Dirichlet series received also much attention after the founding papers of Gordon--Hedenlmann \cite{GH99} and Bayart \cite{Bayart} (see for instance \cite{BQS,QS,Q}). The case of general Dirichlet series is less studied. Nevertheless, when $X$ is some weighted space of {\bf entire} Dirichlet series of type $\Lambda$ satisfying a certain property, properties of composition operators acting on $X$ are quite well understood (see \cite{DMK,HK,HHK}). In \cite{DK20}, using Liouville's theorem, it is proved that if $C_\varphi$ defines a bounded composition operator on a weighted Hilbert space of entire Dirichlet series, then $\varphi$ must be an affine function. \

The aim of this paper is to focus on the situation when our space of Dirichlet series is formed by functions which are holomorphic in {\bf some half-plane} but not necessarily the whole plane. In this context, we could not apply Liouville's theorem. Nevertheless, we will still focus on the case when $\varphi$ is a polynomial (and then necessarily $\varphi$ must be affine), since, for general symbols, we could not really hope that $C_\varphi$ maps a space of Dirichlet series into itself because the restrictions on the sequence $\Lambda$ generating the space will be too severe.\

When $\varphi(z)=az+b$, we will characterize boundedness (see Section 3) and compactness (see Section 4) of $C_\varphi$ on a weighted Dirichlet space $\mathcal H(\beta,\Lambda)$ (see next section for the definition). Then we will also study the dynamics properties (cyclicity and supercyclicity) in Section 5 of $C_\varphi$ on  $\mathcal H(\beta,\Lambda)$. In Section 6 we give some results about complex symmetry. In the final section we end with a note of how our results relate for weighted Hilbert spaces of entire Dirichlet series.

\section{\bf Weighted spaces of holomorphic Dirichlet series}

Fix $\Lambda=(\lambda_n)_{n\geq 1}$ a strictly increasing sequence of non-negative real numbers tending to $\infty$ and satisfying $L<\infty$. To perform the main object of our investigation, we need the following result from \cite{HHK}.
	
\begin{lem}\label{HHKlem}
$L < \infty$ if and only if $\dsum_{n=1}^\infty e^{-r\lambda_n} < \infty$ for all $r > L$. Furthermore, $\dsum_{n=1}^\infty e^{-r\lambda_n} = \infty$ for any $r < L$.
\end{lem}
	
Let $\beta = (\beta_n)_{n\geq 1}$ be a sequence of positive real numbers. The sequence space
\[
\lbsq = \left\{(a_n)_{n\geq 1}\subset\C: \sum_{n=1}^\infty |a_n|^2 \beta_n^2 < \infty\right\}
\]
is a Hilbert space with inner product defined for $a=(a_n)_{n\geq 1}, b=(b_n)_{n\geq 1}\in \lbsq$ by 
\begin{equation}\label{weight-norm}
\langle a,b\rangle = \sum_{n=1}^{\infty}a_n\overline{b}_n \beta_n^2.
\end{equation}
We put
\begin{equation}\label{beta}
\beta_* = \liminf_{n\to\infty}\frac{\log\beta_n}{\lambda_n}.
\end{equation}
The following result serves as an important motivation of our investigation.

\begin{prop}\label{c-w}
Suppose $\beta_*\ne-\infty$. If $f(z) = \dsum_{n=1}^\infty a_n e^{-\lambda_n z}$ is a Dirichlet series of type $\Lambda$ whose coefficients $(a_n)_{n\geq 1}\in\lbsq$, then $f$ converges uniformly on every compact subset of $\C_{\frac L2-\beta_*}$. In particular, $f$ is holomorphic in $\C_{\frac{L}{2}-\beta_*}$.
\end{prop}
	
\begin{proof}
Suppose that $\beta_*\ne \pm\infty$. Fix a compact subset $K$ of $\C_{\frac{L}{2}-\beta_*}$. Then, there exists $\eta>\frac L2-\beta^*$ such that for every $z\in K$, we have $\Re(z)\geq \eta$. Thus
\[
\sum_{n=1}^\infty \sup_{z\in K}\left|a_n e^{-\lambda_n z}\right|\leq \sum_{n=1}^\infty |a_n| e^{-\lambda_n \eta}.
\]
Apply the Cauchy-Schwarz inequality to get
\[
\sum_{n=1}^\infty \sup_{z\in K}\left|a_n e^{-\lambda_n z}\right|\leq  \left(\dsum_{n=1}^\infty |a_n|^2 \beta_n^2\right)^{1/2} \left(\dsum_{n=1}^\infty \frac{e^{-2\lambda_n \eta}}{\beta_n^2}\right)^{1/2}.
\]
Since $\eta > \frac{L}{2} - \beta_*$, we can take $0 < \eps < \delta:= \eta-\bigbracket{\frac{L}{2}-\beta_*}$. By definition of $\beta_*$, there exists $N$ such that for all $n\geq N$,
\[
\frac{1}{\beta_n^2} \leq e^{-2\lambda_n\bigbracket{\beta_* - \frac{\eps}{2}}}.
\]
Applying Lemma \ref{HHKlem} with $r=L+\delta$, we obtain
\[
\sum_{n\geq N} \frac{1}{\beta_n^2}e^{-2\lambda_n \eta}
\leq \sum_{n\geq N} e^{-2\lambda_n\bigbracket{\beta_* - \frac{\eps}{2}}}e^{-2\lambda_n \eta}
\leq \sum_{n\geq N} e^{-\lambda_n(L+\delta)} < \infty,
\]
which proves that
\[
\sum_{n=1}^\infty \sup_{z\in K}\left|a_n e^{-\lambda_n z}\right|<\infty.
\]
Thus, $f$ converges uniformly on $K$. Since this is valid for every compact $K$ in $\C_{\frac L2-\beta_*}$, we deduce that $f$ is analytic on $\C_{\frac L2-\beta_*}$.

The case when $\beta_*=\infty$ follows along the same lines with some tiny adjustments left to the reader.
\end{proof}
	
Proposition \ref{c-w} leads us to the following definition of weighted space of holomorphic Dirichlet series in $\C_{\frac{L}{2}-\beta_*}$:
\[
\Hblam = \left\{f(z) = \sum_{n=1}^\infty a_n e^{-\lambda_n z}: (a_n)_{n\geq 1} \in\lbsq\right\}.
\]
This is a Hilbert space with inner product inherited from \eqref{weight-norm}. More precisely, the inner product on $\Hblam$ is defined as
\begin{equation}\label{eq:inner-product-dirichlet}
\langle f,g\rangle=\sum_{n=1}^{\infty}a_n\overline{b}_n \beta_n^2,
\end{equation}
for every $f(z)=\dsum_{n=1}^\infty a_n e^{-\lambda_n z}, g(z)=\dsum_{n=1}^\infty b_n e^{-\lambda_n z} \in\Hblam$. Note that the inner product \eqref{eq:inner-product-dirichlet} is well defined because of the uniqueness property on coefficients for Dirichlet series of type $\Lambda$, namely if $\dsum_{n=1}^\infty a_n e^{-\lambda_n z}=0$, then $a_n=0$ for all $n\geq 1$, see \cite[page 8]{Valiron}.

\begin{rem}\label{b=-infty}
The assumption $\beta_*>-\infty$ in Proposition \ref{c-w} is important, otherwise $\C_{\frac{L}{2}-\beta_*} = \emptyset$. It is also essential because we can prove that if $\beta_* = -\infty$ then for every $z_0\in\C$ there exists a Dirichlet series of type 
$\Lambda$, $f(z)=\dsum_{n=1}^\infty a_n e^{-\lambda_n z}$ with $(a_n)_{n\geq 1}\in\lbsq$, which does not converge at $z_0$. Indeed, there exists $(n_p)_{p\geq 1}\uparrow\infty$ large enough, such that
\[
\beta_{n_p}^2 < e^{-2\lambda_{n_p} \Re(z_0)},\ \text{for all $p\ge 1$}.
\]
Take $(a_n)_{n\geq 1}$ as follows
\[
a_n=
\begin{cases}
\frac{1}{p} e^{\lambda_{n_p} z_0}, & \text{$n = n_p\ (p=1,2,\ldots)$,}\\
0, & \text{otherwise}.
\end{cases}
\]
Then,
\[
\sum_{n=1}^\infty |a_n|^2\beta_n^2 \leq \sum_{p=1}^\infty \frac{1}{p^2} e^{2\lambda_{n_p} \Re(z_0)} e^{-2\lambda_{n_p} \Re(z_0)} = \sum_{p=1}^\infty \frac{1}{p^2} < \infty.
\]
However, at $z_0$ we have
\[
\sum_{n=1}^\infty a_n e^{-\lambda_n z_0} = \sum_{p=1}^\infty \frac{1}{p} e^{\lambda_{n_p} z_0}e^{-\lambda_{n_p} z_0} = \sum_{p=1}^\infty \frac{1}{p} = \infty.
\]
\end{rem}

Note that for the case $\beta_* = \infty$, we have entire Dirichlet series which have been studied quite well (see, e.g., \cite{DK20} and related references).  {\bf{Therefore, in the sequel we assume that the condition $\beta_*\ne\pm\infty$ always holds.}}


In the rest of the paper, we adopt the following notation. For $n\geq 1$, $z\in\mathbb C$,
\begin{equation}\label{bon}
q_n(z)=\frac{1}{\beta_n}e^{-\lambda_n z}.
\end{equation}

\begin{prop}\label{prop:ONB}
The sequence $(q_n)_{n\geq 1}$ forms an orthonormal basis of $\Hblam$.
\end{prop}

\begin{proof}
It is immediate from \eqref{eq:inner-product-dirichlet}.
\end{proof}

Note that the proof of Proposition~\ref{c-w} shows that for every point $w\in\C_{\frac{L}{2}-\beta_*}$, the evaluation functional $\delta_w$ is continuous on $\Hblam$. Furthermore, using Proposition~\ref{prop:ONB}, we can compute the kernel $k_w$ at point $w\in\C_{\frac{L}{2}-\beta_*}$ by
\begin{align*}
k_w(z)=&\sum_{n=1}^\infty \langle k_w,q_n\rangle q_n(z) = \sum_{n=1}^\infty \overline{q_n(w)}q_n(z)\\
=& \sum_{n=1}^\infty \frac{1}{\beta_n^2} e^{-\lambda_n (z+\overline{w})}.
\end{align*}
Thus we obtain the following result.

\begin{prop}\label{H-repker}
The spaces $\Hblam$ are all reproducing kernel Hilbert spaces with reproducing kernel
\begin{equation}\label{ker-H}
K(z,w) = k_w(z) = \sum_{n=1}^\infty \frac{1}{\beta_n^2} e^{-\lambda_n(z+\overline{w})}, \quad z,w\in\C_{\frac{L}{2}-\beta_*}.
\end{equation}
\end{prop}
In particular, we deduce the norm of the kernel
\begin{equation}
\|k_w\|^2 = \sum_{n=1}^\infty \frac{1}{\beta_n^2} e^{-2\lambda_n\Re(w)}, \quad w\in\C_{\frac{L}{2}-\beta_*}.
\end{equation}

\section{\bf Bounded composition operators induced by a polynomial}
	
Since $\Hblam$ is a functional Hilbert space (in which evaluations are continuous), it follows easily from the closed graph theorem that the space $\Hblam$ is \textit{invariant} under a composition operator $C_\varphi$, i.e. $C_\varphi\big(\Hblam\big)\subseteq \Hblam$,  if and only if $C_\varphi$ is bounded on $\Hblam$.

The study of boundedness is based on the following two simple lemmas. A version of the first one appears in \cite{GH99} but for the completeness we give a proof.
\begin{lem}\label{branch}
Let $f$ be holomorphic in some half-plane $C_\theta=\{z:\Re(z)>\theta\}$ ($\theta\in\mathbb{R}$) and suppose that
\begin{equation}\label{eq:converge-branch}
\lim_{\Re(z)\to\infty}e^{f(z)}=b,
\end{equation}
for some $b\in\mathbb{C^*}:=\mathbb{C}\setminus\{0\}$. Then, there exists a branch of the logarithm $\log_{\alpha}$ and there exists $\ell\in\mathbb Z$ such that
\begin{equation}\label{eq:limite-log}
\lim_{\Re(z)\to\infty}f(z)=\log_{\alpha}(b)+2i\pi \ell.
\end{equation}
\end{lem}

Here we denote $\log_{\alpha}(w)=\log|w|+i\arg_{(\alpha,\alpha+2\pi)}(w)$ for all $w\in\mathbb C\setminus L_\alpha$, and $L_\alpha=\{re^{i\alpha}: r\geq 0\}$.

\begin{proof}
First note that since $b\in\mathbb C^*$, there exists $\alpha\in [0,2\pi)$ such that $b\in \mathbb C\setminus L_\alpha$, and  since $\mathbb C\setminus L_\alpha$ is open, there exists $\delta_1>0$ such that the open disc $D(b,\delta_1)$ is contained in $\mathbb C\setminus L_\alpha$. Now using \eqref{eq:converge-branch}, we can find $R_0>\theta$ such that
\[
\Re(z)>R_0 \Longrightarrow e^{f(z)}\in D(b,\delta_1)\subset \mathbb C\setminus L_\alpha.
\]
In particular, $\log_{\alpha}(e^{f(z)})$ is well defined on $\mathbb C_{R_0}$ and the function
\[
z\longmapsto \frac{\log_{\alpha}(e^{f(z)})-f(z)}{2i\pi}
\]
is holomorphic on $\mathbb C_{R_0}$ and its values are contained in $\mathbb Z$. The open mapping theorem (for holomorphic functions) implies that there exists $\ell\in\mathbb Z$ such that for all $z\in\mathbb C_{R_0}$, we have $\log_\alpha(e^{f(z)})=f(z)-2i\pi\ell$.

Now let $\varepsilon>0$. Then (by continuity of $\log_\alpha$) we can find $\delta_2>0$ such that
\begin{equation}\label{eq2:limite-log}
|w-b|<\delta_2\Longrightarrow |\log_{\alpha}(w)-\log_\alpha(b)|<\varepsilon.
\end{equation}
Using \eqref{eq:converge-branch} once more, we can find $R_1>\theta$ such that
\[
\Re(z)>R_1\Longrightarrow |e^{f(z)}-b|<\delta_2,
\]
and by \eqref{eq2:limite-log}, we get
\[
\Re(z)>R_1\Longrightarrow |\log_\alpha(e^{f(z)})-\log_{\alpha}(b)|<\varepsilon.
\]
We thus deduce that
\[
\Re(z)>\max(R_1,R_0)\Longrightarrow |f(z)-2i\pi\ell-\log_{\alpha}(b)|<\varepsilon,
\]
which gives the result.
\end{proof}

The second lemma is similar to one that appears in \cite{HHK} for the case of entire Dirichlet series.

\begin{lem}\label{e-branch}
Let $f(z)=\dsum_{n=1}^\infty a_n e^{-\lambda_n z}$ be a Dirichlet series with $L < \infty$ and $D \neq \infty$. If $a_m$ is the first non-zero coefficient, then
\[
\lim_{\Re(z)\to\infty} e^{\lambda_m z}f(z) = a_m.
\]
\end{lem}

\begin{proof}
Let $f(z) = \dsum_{n=m}^\infty a_n e^{-\lam_n z}$ be a Dirichlet series where $a_m\neq 0$. Then we have
\begin{equation}\label{nonzero-coff}
e^{\lambda_m z}f(z) = \sum_{n\geq m} a_n e^{-(\lambda_n - \lambda_m) z}.
\end{equation}
Notice that $\lambda_n - \lambda_m \geq 0$ for all $n \geq m$, so that the right-hand side of \eqref{nonzero-coff} is a Dirichlet series of type $(\lambda_n-\lambda_m)_{n\geq m}$. We claim that the associated abscissa of uniform convergence $\sigma_u'$ satisfies $\sigma_u' \neq \infty$.
	
We show that the associated $L'$ and $D'$ to the Dirichlet series above satisfy $L' = L$ and $D' = D$. Indeed, since $\displaystyle \lim_{n\to\infty} \frac{\lambda_n}{\lambda_n - \lambda_m} = 1$, it follows that
\[
L' = \limsup_{n\to\infty} \frac{\log n}{\lambda_n - \lambda_m} = \limsup_{n\to\infty} \bigbracket{\frac{\log n}{\lambda_n} \cdot \frac{\lambda_n}{\lambda_n - \lambda_m}} = L < \infty.
\]
and similarly $D' = D$.

Since $L < \infty$ and $D\neq \infty$, It follows from \eqref{eq:Valiron-formula} that $\sigma_u' \neq \infty$, so that $e^{\lambda_m z}f(z)$ is uniformly convergent in some half-plane. Therefore we may interchange limit and sum to obtain
\begin{align*}
\lim_{\Re(z)\to\infty} e^{\lambda_m z}f(z) &= \lim_{\Re(z)\to\infty} \sum_{n\geq m} a_n e^{-(\lambda_n - \lambda_m) z}\\
&= \sum_{n\geq m} \lim_{\Re(z)\to\infty}\bigbracket{a_n e^{-(\lambda_n - \lambda_m) z}} = a_m.
\end{align*}
\end{proof}

\subsection{Important necessary condition for bounded $C_\varphi$ with general symbols}

In this subsection, we obtain an important necessary condition which a general symbol $\varphi$ inducing a bounded composition operator $C_\varphi$ must satisfy. See \cite{DK20,HHK} for analogous results in the entire case. First, note that in the case when $\beta_*\ne\pm \infty$, all $f\in \Hblam$ converge on the proper right half-plane $\C_{\frac{L}{2}-\beta_*}$ and in particular, satisfy $D_f\ne \infty$ (by Valiron's formulae). Hence Lemma \ref{e-branch} can be applied to elements of $\Hblam$.
	
\begin{prop}\label{gen-symbol}
Let $\beta_*\neq \pm\infty$ and $\varphi$ be a holomorphic self-map on $\C_{\frac{L}{2}-\beta_*}$. Suppose that the operator $C_\varphi$ is bounded on $\Hblam$. Then the following condition holds
\begin{equation}\label{nec-gen}
\lim_{\Re(z)\to\infty} \frac{\varphi(z)}{z} = A,
\end{equation}
where $A$ is a real constant such that
\begin{enumerate}
\item if $\lambda_1 > 0$, then $A \geq 1$, while
\item if $\lambda_1 = 0$, then either $A=0$ or $A\geq 1$.
\end{enumerate}
Moreover, if for each $k\in\N$ we denote by $m_k$ the index of the first non-zero term of $C_\varphi (e^{-\lambda_k z}) = e^{-\lambda_k \varphi(z)}$ in the Dirichlet series representation \eqref{DS}, then
\begin{equation}\label{constA}
\lambda_{m_k}=A \lambda_k,
\end{equation}
for all $k\in\N$.
\end{prop}
	
\begin{proof}
\textbf{- Case 1: $\lambda_1 > 0$.} For each $k\in\N$, $C_\varphi (e^{-\lambda_k z}) \in \Hblam$, and hence we can represent it as
\[
C_\varphi (e^{-\lambda_k z}) = e^{-\lambda_k\varphi(z)} = \sum_{n\geq m_k} b_n^{(k)} e^{-\lambda_n z},
\]
where, by assumption, $m_k$ is the smallest integer such that $b_{m_k}^{(k)} \neq 0$, and $(b_n^{(k)})_n\in\ell_\beta^2$.
		
By Lemma \ref{e-branch}, we have
\[
\lim_{\Re(z)\to\infty}{e^{\lambda_{m_k}z} e^{-\lambda_k\varphi(z)}} = b_{m_k}^{(k)}.
\]
Applying Lemma \ref{branch} to $f(z)=\lambda_{m_k} z - \lambda_k\varphi(z)$, there exist a branch of the logarithm $\log_{\alpha}$ and $\ell\in\mathbb{Z}$, such that
\[
\lim_{\Re(z)\to\infty}\big(\lambda_{m_k} z - \lambda_k\varphi(z)\big) = \log_{\alpha}\big(b_{m_k}^{(k)}\big)+2i\pi \ell.
\]
Thus, we get
\[
\lim_{\Re(z)\to\infty} \left(\frac{\lambda_{m_k}}{\lambda_k} - \frac{\varphi(z)}{z}\right) = 0,
\]
that is
\[
\lim_{\Re(z)\to\infty} {\frac{\varphi(z)}{z}} = \frac{\lambda_{m_k}}{\lambda_k} \ge 0.
\]
This shows that $\dfrac{\lambda_{m_k}}{\lambda_k}$ is independent of $k$, and hence we can denote this non-negative constant as $A$.

In particular, for $k=1$, we have $A = \dfrac{\lambda_{m_1}}{\lambda_1} \geq 1$, because $m_1 \geq 1$.

\textbf{- Case 2: $\lambda_1 = 0$.} In this case, replacing $k\in\N$ by $k\ge 2$ in the proof for Case 1 above, we still have a non-negative constant $A = \dfrac{\lambda_{m_k}}{\lambda_k}$, for all $k\geq 2$. In particular, for $k=2$, we have
\[
A = \frac{\lambda_{m_2}}{\lambda_2} =
\begin{cases}
0, & \text{if $m_2 = 1$ (because $\lambda_1=0$)},\\
\geq 1, & \text{if $m_2\geq 2$ (because $\lambda_{m_2} \ge \lambda_2$)}.
\end{cases}
\]
Finally, in that case, we have $m_1=1$ and thus $\lambda_{m_1}=0=A\lambda_1$.
\end{proof}

\begin{rem}
If $A\geq 1$, it follows immediately from \eqref{constA} that the map $k\longmapsto m_k$ is strictly increasing.
\end{rem}

\subsection{Necessary conditions for boundedness of $C_\varphi$: a polynomial symbol}

In case the symbol $\varphi$ is a polynomial, we can get a crucial information: $\varphi$ must be an affine function.

To study this case, for a composition operator to be well-defined, we first need the following result.

\begin{lem}\label{affine-sel-map}
An affine function $az+b$ ($a,b\in\mathbb{C}$) is a self-map of the half-plane $\mathbb{C}_\theta$ ($\theta\in\mathbb{R}$) if and only if
\begin{itemize}
\item[(i)] $a\in\R^+$ and $\Re(b) \geq (1-a)\theta$, or
\item[(ii)] $a = 0$ and $\Re(b) > \theta$.
\end{itemize}
\end{lem}
	
\begin{proof}
Suppose $az+b$ is a self-map on $\C_\theta$. If $a = 0$, then it clear that $\Re(b) > \theta$, because $b\in\C_\theta$. So we get (ii).

In case $a\ne0$, we show that $a\in\R^+$. Assume to the contrary that $a\not\in\R^+$, which means that $a=|a|e^{i\theta_1}$, with $\theta_1\neq 0 \mod(2\pi)$. For every $z=x+i y$, $x,y\in\R$, we have $\Re(az)=x|a|\cos(\theta_1)-y|a|\sin(\theta_1)$.

\textbf{Case 1:} If $\theta_1=\pi \mod(2\pi)$, then $\Re(az)=-x|a|$, and so letting $\Re(z)=x\to \infty$, we get that $\Re(az+b)\to -\infty$, which contradicts that $az+b$ is a self-map of $\C_{\theta}$.

\textbf{Case 2:} If $\theta_1\neq \pi \mod(2\pi)$, then $\sin(\theta_1)\neq 0$. Thus if we fix $x > \theta$ and let $\Im(z) = y\to \infty$ (if $\sin(\theta_1)>0$) or $\Im(z) = y\to -\infty$ (if $\sin(\theta_1)<0$), we get that $\Re(az+b)\to -\infty$, which also contradicts the assumption.

Thus we have $a\in\R^+$. Then for any $\eps > 0$, a point $z \in \mathbb{C}_\theta$ with $\Re(z) = \theta + \eps$ satisfies $\Re(az + b) = a\theta + a\eps + \Re(b) > \theta$. This means $\Re(b) > (1-a)\theta - a\eps$. Letting $\eps\to0$, we obtain (i).

Conversely, if (ii) holds, i.e. $a=0$ and $\Re(b) > \theta$, then we are obviously done. If (i) holds, i.e. $a\in\R^+$ and $\Re(b) \geq (1-a)\theta$, then for any $z\in\C_\theta$, we have $\Re(az+b)=a\Re(z)+\Re(b)>a\theta+\Re(b)\geq a\theta+(1-a)\theta=\theta$, which shows that $az+b\in\mathbb{C}_\theta$. Thus $az+b$ is a self-map of $\mathbb{C}_\theta$.
\end{proof}

Now we have the following result.

\begin{thm}\label{affine-symbol}
Let $\beta_*\neq \pm\infty$ and $\varphi$ be a self-map polynomial of $\mathbb{C}_{\frac{L}{2}-\beta_*}$. Suppose that the operator $C_\varphi$ is bounded on $\Hblam$. Then the following assertions hold:
\begin{enumerate}
\item If $\lambda_1 > 0$, then $\varphi(z) = az+b,\ a\geq 1,\ b\in\C$.
\item If $\lambda_1 = 0$, then either $\varphi(z) = b,\ b\in\C$ (that is $a=0$), or $\varphi(z) = az+b,\ a\geq 1,\ b\in\C$.
\end{enumerate}
Moreover, in addition, $a$ satisfies the condition
\begin{equation}\label{const-a}
\lambda_{m_k}=a \lambda_k,\quad\text{for all $k\in\N$},
\end{equation}
where $m_k$ is the index of the first non-zero term of $C_{az+b}(e^{-\lambda_k z}) = e^{-\lambda_k (az+b)}$ in the Dirichlet series representation \eqref{DS}, and $b$ satisfies either of the following conditions:
\begin{equation}\label{a,b-1}
\Re(b) > \frac{L}{2}-\beta_*, \quad a=0,
\end{equation}
\begin{equation}\label{a,b-2}
\Re(b) \ge (1-a)\Big(\frac{L}{2}-\beta_*\Big), \quad a\ge1.
\end{equation}
\end{thm}
	
\begin{proof}
By Proposition \ref{gen-symbol}, we have $\displaystyle \lim_{\Re(z)\to\infty} \frac{\varphi(z)}{z}$ is a (real) constant, which implies that $\deg(\varphi)\le 1$. So $\varphi(z) = az + b$, where $a$ and $b$ are complex numbers. Then again by Proposition \ref{gen-symbol}, we get $a\in\mathbb{R}$ and moreover, if $\lambda_1 > 0$, then $a \geq 1$, while if $\lambda_1 = 0$, then either $a = 0$ or $a\geq 1$, and also $a$ satisfies condition \eqref{const-a}. Furthermore, \eqref{a,b-1} and \eqref{a,b-2} follow from Lemma \ref{affine-sel-map} with $\theta=\dfrac{L}{2}-\beta_*$.
\end{proof}

From now on, an affine symbol $\varphi$ stated in Theorem \ref{affine-symbol} is supposed to be given.

\subsection{Two trivial cases of the affine symbols ($a=0,1$)}

In the following, we denote by $\|\cdot\|_{\rm op}$ the operator norm.

\begin{prop}\label{bdd-a0}
Let $\beta_*\neq \pm\infty$ and $b\in\C_{\frac{L}{2}-\beta_*}$. Then $C_b$ is a bounded composition operator on $\Hblam$ if and only if $\lambda_1=0$. Moreover, in that case, we have $\|C_b\|_{\rm op} = \beta_1\|k_b\|$.
\end{prop}

\begin{proof}
Suppose that $C_b$ is a bounded on $\Hblam$. By Theorem \ref{affine-symbol}, $\lambda_1$ must necessarily be zero.

Conversely, suppose that $\lambda_1=0$. It implies that the function $\beta_1q_1(z) = e^{-\lambda_1z}=1$ belongs to $\Hblam$. (Recall that the functions $q_k$ are given by \eqref{bon}.) Now, note that for all $f\in\Hblam$, we have
\[
(C_b f)(z)=f(b)=\langle f,k_b\rangle \beta_1q_1(z),
\]
Hence $C_b = \beta_1q_1 \otimes k_b$ is  a rank one operator. In particular, it is bounded and $\|C_b\|_{\text{op}} = \beta_1\|q_1\| \|k_b\|$. But $\|q_1\|=1$, which implies that $\|C_b\|_{\text{op}} = \beta_1\|k_b\|$.
\end{proof}

\begin{prop}\label{bdd-a1}
Let $\beta_*\neq \pm\infty$. Then $C_{z+b}$ is a bounded composition operator on $\Hblam$ if and only if $\Re(b)\ge0$. Moreover, in that case, we have $\|C_{z+b}\|_{\rm op} = e^{-\lam_1 \Re(b)}$.
\end{prop}
	
\begin{proof}
Suppose $C_{z+b}$ is bounded. Then by \eqref{a,b-2}, we get immediately that $\Re(b)\geq 0$. Conversely, suppose that $\Re(b)\geq 0$. Note that for every $k\in\mathbb N$, $C_{z+b}q_k=e^{-\lambda_k b}q_k$. Thus $C_{z+b}$ is a diagonal operator with a sequence of eigenvalues equal to $(e^{-\lambda_k b})_{k\geq 1}$. Using now that $\Re(b)\geq 0$, we see that this sequence is decreasing in modulus, and so it is well-known (and easy to see) that $C_{z+b}$ is bounded and $\|C_{z+b}\|_{\rm op}=e^{-\lam_1 \Re(b)}$.
\end{proof}

\subsection{Characterization for boundedness of $C_{az+b}\ (a\ne0,1)$}

Throughout this subsection, we always assume that an affine symbol $\varphi(z)=az+b$, with $a>1$ and $b\in\mathbb{C}$, is given. 

We need some supplementary notation.

\begin{defn}\label{ratiosetdef}
For a given sequence of real frequencies $\Lam = (\lam_n)_{n\geq 1}$, define the set
\[
\mathcal{R}(\Lam) = \left\{r\in[1,\infty): \forall n\in\N,\ \exists m=m_n \geq n,\ r\lam_n = \lam_m \right\}.
\]
\end{defn}
This set, by Theorem \ref{affine-symbol}, is precisely the set of all possible values of $a$ for which a non-trivial symbol $\varphi(z)=az+b$ induces a bounded composition operator $C_\varphi$ on the space $\mathcal{H}(\beta,\Lambda)$.

\begin{rem}\label{ratiosetrem}
Notice that we always have $1\in \mathcal{R}(\Lam)$. Furthermore, since $(\lam_n)_{n\geq 1}$ is a strictly increasing sequence, for a given $r\in \mathcal{R}(\Lam)$, to each $n$ there corresponds a unique $m_n\geq n$, such that $r\lam_n = \lam_{m_n}$.
\end{rem}
	
We put $\mathcal{R}_1(\Lam) = \mathcal{R}(\Lam)\setminus\{1\}$. Depending on the given sequence $(\lambda_n)_{n\geq 1}$, it may happen that $\mathcal{R}(\Lambda)=\{1\}$, i.e. $\mathcal{R}_1(\Lam) = \emptyset$, as well as $\mathcal{R}_1(\Lam) \ne \emptyset$. The following examples are taken from \cite{HK}.

\begin{exa}\label{a=1}
Let $\lam_n=n!$. Then $\mathcal{R}(\Lambda)$ is the singleton $\{1\}$.
\end{exa}

\begin{exa}\label{many-a}
(1) For $\lam_n=\log n$ (the classical Dirichlet series) or $\lam_n=n$, every $\ell\in\mathbb{N}$ belongs to $\mathcal{R}(\Lam)$.

(2) Consider a geometric sequence $(\lam_n)$ with the ratio $q>1$, given by $ \lam_1>0$ and $\lam_{n}=\lam_1q^{n-1}$, $n\geq1$. In this case any value $q^\ell \ (\ell\in\mathbb{N})$ belongs to $\mathcal{R}(\Lam)$.
\end{exa}

Now let $a\in \mathcal R_1(\Lam)$ be given. In principle, for each $n\in\N$, the index $m_n \ge n$ for which $\lam_{m_n} = a\lam_n$, depends on $a$, i.e. $m_n=m_n^{(a)}$ (note that by Propositions \ref{gen-symbol}, if $\lambda_1 = 0$, then $m_1 = 1$). To simplify the expositions, in what follows, we skip the superscript $(a)$ whenever there is no confusion in context.

For $n\in\N$ and $x\in\R$, we also define the quantity
\[
r_n(a,x) = r_n(\Lam,\beta,a,x) := e^{-\lam_n x}\frac{\beta_{m_n}}{\beta_n}.
\]

Now we are able to state and prove the following boundedness criterion for the case $a>1$.

\begin{prop}\label{cri-bdd-a>1}
Let $\beta_*\neq \pm\infty$, $a>1$ and $b\in\C$. Then $C_{az+b}$ is a bounded composition operator on $\Hblam$ if and only if the following conditions are satisfied
		\begin{enumerate}
			\item $a\in \mathcal R_1(\Lambda)$,
			\item $\Re(b) \geq (1-a)\Big(\dfrac{L}{2}- \beta_*\Big)$,
			\item the sequence $\Big(r_n(a,\Re(b))\Big)_{n\ge1}$ is bounded.
		\end{enumerate}
Moreover, in this case, 		
\[
\|C_{az+b}\|_{{\rm op}} = \sup_{n\in\N} r_n(a,\Re(b)).
\]
\end{prop}
	
\begin{proof}
\textit{Necessity}. Suppose $C_{az+b}$ is a bounded operator on $\Hblam$. Then conditions (1) and (2) follow from Theorem \ref{affine-symbol}. It remains to show (3).
		
There is some constant $M>0$ such that $\|C_{az+b}f\| \leq M \|f\|$ for all $f\in\Hblam$. In particular, for probe functions $q_n(z) = \frac{1}{\beta_n} e^{-\lambda_n z}$, by (1) we have $C_{az+b}q_n=\frac{\beta_{m_n}}{\beta_n}e^{-\lambda_n b}q_{m_n}$, and thus
\[
\|C_{az+b} q_n\| = e^{-\lambda_n \Re(b)}\cdot \frac{\beta_{m_n}}{\beta_n}\leq M \|q_n\| = M,\quad\text{for all $n\in\N$},
\]
which gives (3).

\noindent \textit{Sufficiency}. Conversely, suppose all three conditions (1) -- (3) are satisfied. Note that (2) guarantees, by Lemma \ref{affine-sel-map}, that $z\mapsto az+b$ is a self-map of $\displaystyle \mathbb{C}_{\frac{L}{2}-\beta_*}$. Also, (3) shows that there exists $M > 0$ such that $0< r_n(a,\Re(b))\leq M$, for all $n$. Let $f(z) = \dsum_{k=1}^\infty a_n e^{-\lambda_n z}\in\Hblam$. Hence, by (1),
\begin{eqnarray*}
\|C_{az+b} f\|^2%
&=& \sum_{n=1}^\infty |a_ne^{-\lambda_n b}|^2 \beta_{m_n}^2 = \sum_{n=1}^\infty |a_n|^2 \beta_n^2 r_n(a,\Re(b))^2\\
&\le& M^2 \sum_{n=1}^\infty |a_n|^2 \beta_n^2 = M^2 \|f\|^2,
\end{eqnarray*}
which shows that $C_{az+b}$ is bounded on $\Hblam$.

Moreover, from proofs of both necessity and sufficiency it follows that $\displaystyle \|C_{az+b}\|_{{\rm op}} = \sup_{n\in\N} r_n(a,\Re(b))$.
\end{proof}
	
Combining Propositions \ref{bdd-a0}, \ref{bdd-a1} and \ref{cri-bdd-a>1}, we obtain a characterization of boundedness for $C_{az+b}$ on $\Hblam$.

\begin{thm}\label{cri-bdd}
Let $\beta_*\neq \pm\infty$ and $\varphi(z)=az+b$ ($a,b\in\mathbb{C}$). Consider the following statements.
\begin{enumerate}
    \item $\varphi(z) = b$ for some $b\in\C$ with $\Re(b) > \dfrac{L}{2} - \beta_*$.
    \item $\varphi(z) = az+b$, where
    \[
    \begin{cases}
		a\in \mathcal R(\Lambda),\\
		\Re(b) \geq (1-a)\Big(\dfrac{L}{2}- \beta_*\Big),\\
		\text{the sequence $\Big(r_n(a,\Re(b))\Big)_{n\ge1}$ is bounded}.
		\end{cases}
    \]
\end{enumerate}
The following are true for a composition operator $C_\varphi$ acting on the space $\Hblam$.
\begin{enumerate}
    \item[(i)] If $\lambda_1 = 0$, then $C_\varphi$ is bounded if and only if either (1) or (2) holds.
    \item[(ii)] If $\lambda_1 > 0$, then $C_\varphi$ is bounded if and only if (2) holds.
\end{enumerate}
\end{thm}
	
\section{\bf Essential norm, compactness, Schatten class and compact differences}

\subsection{Essential norm and compactness}

Compactness of a bounded composition operator on a Hilbert space $\Hblam$ can be investigated in different ways. The first makes use of a compactness criterion which states that a bounded linear operator $T$ on $\Hblam$ is compact if and only if for any sequence $(f_n)$ from $\Hblam$ which is weakly convergent to $0$, the sequence $(T f_n)$ converges strongly to $0$ in $\Hblam$. The other way is via the essential norm of $T$ defined by
\begin{align*}
\|T\|_e = \inf\{\|T-K\|_{\text{op}}: K \ \text{is a compact operator on $\Hblam$}\}.
\end{align*}
Clearly, $T$ is compact if and only if $\|T\|_e=0$.

For a bounded composition operator $C_{az+b}$, in case $a=0$, by Proposition \ref{bdd-a0}, we have the following simple result.

\begin{prop}\label{cpt-a0}
Let $\beta_*\neq \pm\infty$. For any $b\in\mathbb{C}_{\frac{L}{2}-\beta_*}$, a composition operator $C_b$ is always compact on the space $\Hblam$.
\end{prop}

\begin{proof}
As noted in Proposition \ref{bdd-a0}, $C_b$ has a rank one and hence it is compact.
\end{proof}

\begin{prop}\label{cpt-a1}
Let $\beta_*\neq \pm\infty$ and let $C_{z+b}$ be a bounded composition operator on $\Hblam$. Then
\[
\|C_{z+b}\|_e=\begin{cases}
1 &\mbox{if }\Re(b)=0\\
0 &\mbox{if }\Re(b)>0.\\
\end{cases}
\]
In particular, $C_{z+b}$ is a compact operator on  $\Hblam$ if and only if $\Re(b)>0$.
\end{prop}

\begin{proof}
Proposition~\ref{bdd-a1} shows that $\Re(b)\geq 0$. Furthermore, as already noticed, $C_{z+b}$ is a diagonal operator with a sequence of eigenvalues equal to $(e^{–\lambda_k b})_{k\geq 1}$. But then it is well-known (see for instance \cite[Problem 171]{Halmos}) that $\|C_{z+b}\|_e=\lim\limits_{n\to \infty}e^{-\lambda_n\Re(b)}$, which gives the conclusion. 
\end{proof}

Proposition \ref{cpt-a1} can also be obtained from the following results for $a\ge1$.

\begin{thm}\label{ess-norm}
Let $\beta_*\neq \pm\infty$, $a\ge1$ and $C_{az+b}$ be a bounded composition operator on the space $\Hblam$. Then $\displaystyle \|C_{az+b}\|_e = \limsup_{n\to\infty} r_n(a,\Re(b))$. In particular, $C_{az+b}$ is compact if and only if $\displaystyle \lim_{n\to\infty} r_n(a,\Re(b))=0$. 	
\end{thm}

\begin{proof}
We follow the standard technique (see, e.g., \cite{DK16}).

Note that by Theorem \ref{cri-bdd}, the sequence $\Big(r_n(a,\Re(b))\Big)_{n\ge1}$ is bounded.

\noindent{$\bullet$ \bf Upper bound}. We use compact (finite rank) operators on $\Hblam$ defined by
\[
K_N\colon f(z)=\sum_{n=1}^\infty a_ne^{-\lambda_n z} \longmapsto \sum_{n=1}^N a_ne^{-\lambda_n z}\quad (N\in\mathbb{N}).
\]
As the $C_{az+b} K_N$ are also compact, we have
\[
\|C_{az+b}\|_e \leq \inf_{p\ge1} \|C_{az+b} - C_{az+b}K_N\|_{\text{op}}.
\]
Observe that for an arbitrary $\displaystyle f(z) = \sum_{n=1}^\infty a_n e^{-\lam_n z}\in\Hblam$ we have
\begin{eqnarray*}
&& \|(C_{az+b} - C_{az+b}K_N)f\|^2 = \Big\|\sum_{n=N+1}^\infty a_ne^{-\lambda_n (az+b)}\Big\|^2\\
&& = \sum_{n=N+1}^\infty |a_n|^2 e^{-2\lam_n \Re(b)} \beta_{m_n}^2 \leq \sup_{n\geq p+1} \Big(e^{-\lam_n \Re(b)}\frac{\beta_{m_n}}{\beta_n}\Big)^2 \sum_{n=N+1}^\infty |a_n|^2 \beta_n^2\\
&& \le \|f\|^2 \sup_{n\geq p+1} r_n(a,\Re(b)).
\end{eqnarray*}
Thus
\[
 \|C_{az+b}\|_e  \leq \sup_{n\geq N+1} r_n(a,\Re(b)), \quad\text{for every $N\ge1$}.
\]
Letting $N\to\infty$, we get
\[
\|C_{az+b}\|_e  \le \limsup_{n\to\infty} r_n(a,\Re(b)).
\]

\noindent{$\bullet$ \bf Lower bound}. Let $K$ be an arbitrary compact operator on $\Hblam$. Consider the sequence of probe functions $(q_n)_{n\geq 1}$, whose norms are all $1$. Since it converges weakly to $0$ (because it is an orthonormal basis), $\displaystyle \lim_{n\to\infty}\|K q_n\|=0$. Hence
\[
\|C_{az+b}-K\|_{\text{op}} \ge \|(C_{az+b}-K) q_n\| \ge \|C_{az+b} q_n\| - \|K q_n\| \quad\text{($n\ge1$)},
\]
from which it follows that
\begin{eqnarray*}
\|C_{az+b}-K\|_{\text{op}}%
&\ge& \limsup_{n\to\infty} \left(\|C_{az+b}q_n\| - \|K q_n\|\right) = \limsup_{n\to\infty} \|C_{az+b}q_n\| \\
&=& \limsup_{n\to\infty} \left(e^{-\lambda_n \Re(b)}\frac{\beta_{m_n}}{\beta_n}\right)
= \limsup_{n\to\infty} r_n(a,\Re(b)).
\end{eqnarray*}
Taking the infimum over all compact operators $K$ on $\Hblam$, we obtain
\[
\|C_{az+b}-K\|_e \ge \limsup_{n\to\infty} r_n(a,\Re(b)).
\]
\end{proof}

\subsection{Schatten class}

Recall that a bounded linear operator $T$ on $\Hblam$ is called a Hilbert--Schmidt operator if it has finite Hilbert--Schmidt norm $\|T\|_{\rm HS}$, which means that for some orthonormal basis $(e_n)_{n\ge1}$ of $\Hblam$, we have
\[
\|T\|_{\rm HS}:=\Big(\sum_{n\in\mathbb{N}} \|T e_n\|^2\Big)^{1/2} < +\infty.
\]
It is well known that $\|T\|_{\rm HS}$ does not depend on the choice of the orthonormal basis and that if $T$ is Hilbert--Schmidt, then it is compact.
	
Furthermore, for $0<p<\infty$, the Schatten $p$-class consists of all bounded linear operators $T$ on $\Hblam$ for which $(T^{*}T)^{p/4}$ is a Hilbert--Schmidt operator (here $T^*$ is the adjoint operator of $T$). The set of Schatten $p$-class operators forms an ideal in the algebra of all bounded linear operators on $\Hblam$. If $T$ is diagonal with respect to an orthonormal basis $(e_n)_{n\ge1}$, that is, $T e_n = a_n\,e_n$ for all $n\ge1$, then it is well known that $T$ belongs to the Schatten $p$-class if and only if $\displaystyle \sum_{n=1}^\infty |a_n|^p < \infty$. Some developments on operators in Schatten classes can be found in \cite{HKZ}.

We study a Schatten class membership of $C_{az+b}$. For a bounded operator $C_{az+b}$ on $\Hblam$, since $C_{az+b}q_n=\frac{\beta_{m_n}}{\beta_n}e^{-\lambda_n b}q_{m_n}$, therefore $\|C_{az+b} q_n\| = e^{-\lambda_n \Re(b)}\cdot \frac{\beta_{m_n}}{\beta_n} = r_n(a,\Re(b))$, and we get an immediate result about its Hilbert--Schmidt property.

\begin{prop}\label{HS} Let $\beta_*\neq\pm\infty$. 
A bounded composition operator $C_{az+b}$ on $\Hblam$ is a Hilbert--Schmidt operator if and only if
\[
\sum_{n=1}^\infty r_n (a,\Re(b))^2 < \infty.
\]
\end{prop}

Concerning the Schatten $p$-class membership, for two trivial cases $a=0$ and $a=1$, some results can be obtained easily.

\begin{prop}\label{Sp-a0} Let $0<p<\infty$. A bounded composition operator $C_b$ on $\Hblam$ belongs to the Schatten $p$-class.
\end{prop}

\begin{proof}
	Since $C_b$ is of rank one, in particular it is of finite rank and so has finitely many non-zero singular values.
\end{proof}

Also for the case $a=1$, as noted in Proposition \ref{bdd-a1}, $C_{z+b}$ is a diagonal operator with a sequence of eigenvalues $(e^{-\lambda_n b})_{n\geq 1}$ corresponding to eigenvectors $(q_n)_{n\ge1}$, and thus we have the following result.

\begin{prop}\label{Sp-a1}  Let $0<p<\infty$ and $\beta_*\neq\pm\infty$.  A bounded composition operator $C_{z+b}$ on $\Hblam$ belongs to the Schatten $p$-class if and only if
\[
\sum_{n=1}^\infty e^{-p\lambda_n \Re(b)} < \infty.
\]
In particular, $C_{z+b}$ belongs to the Schatten $p$-class if $p\Re(b) > L$, and does not belong if $p\Re(b) < L$, where $L$ is defined in \eqref{L}.
\end{prop}
Note that the second statement of the proposition above follows from Lemma \ref{HHKlem}.\\

To go further and get a similar result for $a>1$, we need to compute the adjoint of $C_{az+b}$, which can be done easily. Indeed, for any $\displaystyle g(z) = \sum_{n=1}^\infty a_n e^{-\lambda_n z} \in \Hblam$, we write $\displaystyle(C_{az+b}^*g)(z) = \sum_{n=1}^\infty d_n e^{-\lambda_n z}$ and have
\begin{equation}\label{d_n}
d_n \beta_n = \langle C_{az+b}^*g, q_{n}\rangle = \langle g, C_{az+b} q_n\rangle
= \Big\langle \sum_{k=1}^\infty a_k e^{-\lambda_k z}, C_{az+b} q_n\Big\rangle.
\end{equation}
The terms $\langle e^{-\lambda_k z}, C_{az+b} q_n \rangle$ are computed in two cases:

- Case 1: $a=0$. By Proposition \ref{bdd-a0}, $\lambda_1=0$ and hence
\[
\langle e^{-\lambda_k z}, C_b q_n \rangle =
\begin{cases}
\frac{e^{-\lam_n \overline{b}}}{\beta_n} \beta_1^2, & k=1\\
0, & k > 1.
\end{cases}
\]

- Case 2: $a\ge1$. By Theorem \ref{affine-symbol}, we have
\[
\langle e^{-\lambda_k z}, C_{az+b} q_n \rangle =
\begin{cases}
\frac{e^{-\lam_n \overline{b}}}{\beta_n}\beta_{m_n}^2, & m_n = k\\
0, & m_n \ne k.
\end{cases}
\]

Substituting back these equations into \eqref{d_n} yields the following result.

\begin{prop}\label{adj-a0} Let $\beta_*\neq\pm\infty$, and let $C_{az+b}$ be a bounded composition operator on $\Hblam$. If $\displaystyle f(z) = \sum_{n=1}^\infty a_n e^{-\lambda_n z}\in\Hblam$, then
\begin{equation}\label{C-adj}
(C_{az+b}^* f)(z)=
\begin{cases}
a_1 \beta_1^2 \dsum_{n=1}^\infty \dfrac{e^{-\lam_n \overline{b}}}{\beta_n^2} e^{-\lam_n z}, & a=0\\
\dsum_{n=1}^\infty a_{m_n} \dfrac{\beta_{m_n}^2 e^{-\lam_n \overline{b}}}{\beta_n^2} e^{-\lam_n z}, & a\ge1.
\end{cases}
\end{equation}
\end{prop}

\begin{prop}\label{eigenvalueprop}
	Let $a\geq 1$. The eigenvalues of $C_{az+b}^* C_{az+b}$ are precisely $r_k(a,\Re(b))^2$, $k\in\N$.
\end{prop}

\begin{proof}
	We appeal to Proposition \ref{adj-a0} which provides an explicit formula for the adjoint. Let $g_k = \beta_k q_k$ for all $k\in\N$. We have $(C_{az+b} g_k)(z) = e^{-\lam_k b} e^{-\lam_{m_k} z}$. It follows that
	\begin{align*}
	(C_{az+b}^* C_{az+b} g_k)(z) &= e^{-\lam_k b} \frac{\beta_{m_k}^2}{\beta_k^2} e^{-\lam_k \ol{b}} e^{-\lam_k z} = e^{-2\lam_k \Re(b)} \frac{\beta_{m_k}^2}{\beta_k^2} g_k(z)\\
	&= r_k(a,\Re(b))^2 g_k(z).
	\end{align*}
	
	Now, recall that $(g_k)_{k\geq 1}$ forms an orthogonal basis of $\Hblam$. It follows that the only eigenvalues are precisely those corresponding to these vectors, i.e. precisely $r_k(a,\Re(b))^2$, $k\in\N$.
\end{proof}

As $r_k(a,\Re(b)) > 0$ for every $k$, we have the following corollary.

\begin{cor}
	The eigenvalues of $|C_{az+b}|$ are precisely the values $r_k(a,\Re(b))$, $k\in\N$.
\end{cor}

As a consequence by direct substitution, we thus have the following theorem.

\begin{thm}\label{affineschatten}
	Let $0 < p < \infty$ and $a\geq 1$. Then, $C_{az+b}$ is a $p$th Schatten class operator if and only if
	\begin{align*}
	\dsum_{k=1}^\infty r_k(a,\Re(b))^p < \infty.
	\end{align*} 
\end{thm}

\begin{rem}
Note that the series 
\[
h_p(z) = \sum_{k=1}^\infty \frac{\beta_{m_k}^p}{\beta_k^p} e^{-p\lam_k z},
\]
is a Dirichlet series  of type $(p\lam_n)$, which can be considered as a ``complex version'' of the series $\dsum_{k=1}^\infty r_k(a,x)^p$ when we replace $x\in\mathbb{R}$ by $z\in\mathbb{C}$. By \cite[Chapter II, Theorem 8]{HR15} the computation of $\sigma_c$ is
\begin{align*}
\sigma_c = \limsup_{n\to\infty} \frac{\log \bigbracket{\dsum_{k=1}^n \frac{\beta_{m_k}^p}{\beta_k^p}}}{p\lam_n}.
\end{align*}
Thus, if $\Re(b)$ is larger than this value then $h_p(\Re(b))$ converges and $C_{az+b}$ is Schatten $p$-class. If $\Re(b)$ is smaller, then $C_{az+b}$ is not Schatten $p$-class.
\end{rem}

\begin{rem}
Note also that it is possible that the membership to the $p$-th Schatten class of $C_{az+b}$ does not rely on $p$ at all. Pick $a > 1$ and $\Lam = (a^k)_{k\geq 1}$. In this case, $L = 0$. Then pick $\beta = \displaystyle \left(\prod_{i=1}^{k-1} e^{a^i}\right)_{k\geq 1}$, for which $\beta_* = \dfrac{1}{a^2-a}$ and $C_{az+b}$ is a well-defined bounded composition operator on $\Hblam$. The real function $\displaystyle h_p(x) = \sum_{k=1}^\infty e^{pa^k} e^{-pa^k x} = \sum_{k=1}^\infty e^{-pa^k(x-1)}$ converges for all $x > 1$ and diverges everywhere else, regardless of the value of $p$. It follows that $C_{az+b}$ is a Schatten $p$-class operator if and only if $\Re(b) > 1$.
\end{rem}

\subsection{Compact differences}

In this section we determine when a difference of two bounded composition operators $C_{\vphi_1} - C_{\vphi_2}$ is a compact operator. Recall that the set of compact operators form a vector space. Hence by Proposition \ref{cpt-a0} and Theorem \ref{ess-norm}, it suffices to consider only the case when $\vphi_1$ and $\vphi_2$ are non-constant and the associated sequences $(r_n)$ do not have limit $0$.

\begin{lem}\label{compactdiffaffineprop}
	Let $\beta_* \neq \pm\infty$, $a, a'\geq 1$. Let $C_{az + b}$ and $C_{a'z + b'}$ be bounded composition operators on $\Hblam$. Assume that
\[
\limsup_{n\to \infty}r_n(a,\Re(b))>0\quad \mbox{and} \quad \limsup_{n\to \infty}r_n(a',\Re(b'))>0.
\]
If $C_{az + b} - C_{a'z + b'}$ is compact, then $a=a'$.
\end{lem}

\begin{proof}
	Consider the sequence of probe functions $(q_k)_{k\geq 1}$. By Proposition \ref{prop:ONB}, $q_k\rightharpoonup 0$. Since $C_{az + b} - C_{a'z + b'}$ is compact, therefore $\|(C_{az + b} - C_{a'z + b'})q_k\| \to 0$. Assume to the contrary that $a \neq a'$. Now, since for each $k$ we have $m_{k}^{(a)} \neq m_{k}^{(a')}$, therefore we have
	\begin{align*}
	\|(C_{az + b} - C_{a'z + b'})q_k\|^2 &= \bignorm{\frac{1}{\beta_k} e^{-\lambda_k b}e^{-\lambda_{m_{k}^{(a)}} z} - \frac{1}{\beta_k} e^{-\lambda_k b'}e^{-\lambda_{m_{k}^{(a')}} z}}^2\\
	&= \frac{1}{\beta_k^2} e^{-2\lambda_k \Re (b)}\beta_{m_{k}^{(a)}}^2 + \frac{1}{\beta_k^2} e^{-2\lambda_k \Re (b')}\beta_{m_{k}^{(a')}}^2\\
	&= r_k(a, \Re(b))^2 + r_k(a', \Re(b'))^2,
	\end{align*}
which gives the desired contradiction.
\end{proof}

\begin{prop}\label{prop:essentiel-norme-difference}
Let $\beta_*\neq\pm\infty$, $a\geq 1$. Suppose that $C_{az+b}$ and $C_{az+b'}$ are bounded composition operators on $\Hblam$. Then
 $$
 \|C_{az+b} - C_{az+b'}\|_e = \limsup\limits_{k\to\infty} \dfrac{\beta_{m_k}}{\beta_k}\bigabs{e^{-\lambda_k b} - e^{-\lambda_k b'}}.
 $$
\end{prop}

\begin{proof}
	As before we define the finite-rank (and compact) partial sum operator $K_N$ (see the proof of Theorem~\ref{ess-norm}). Then $(C_{az+b} - C_{az+b'}) K_N$ is compact and
	\begin{align*}
	\|C_{az+b} - C_{az+b'}\|_e \leq \|(C_{az+b} - C_{az+b'}) (I-K_N)\|_{\text{op}},
	\end{align*}
	where $I$ is the identity operator on $\Hblam$.
	
	Let $f(z) = \dsum_{n=1}^\infty a_n e^{-\lam_n z}\in\Hblam$. We then have
	\begin{align*}
	\|(C_{az+b} - C_{az+b'})(I-K_N)(f)\|^2 &= \bignorm{\sum_{n=N+1}^\infty a_n e^{-\lam_n b}e^{-\lam_{m_n} z} - \sum_{n=N+1}^\infty a_n e^{-\lam_n b'}e^{-\lam_{m_n} z}}^2\\
	&= \sum_{n=N+1}^\infty |a_n|^2 \bigabs{e^{-\lam_n b} - e^{-\lam_n b'}}^2 \beta_{m_n}^2\\
	&\leq \sup_{n\geq N+1} \dfrac{\beta_{m_n}^2}{\beta_n^2}\bigabs{e^{-\lambda_n b} - e^{-\lambda_n b'}}^2 \|f\|^2.
	\end{align*}
	It therefore follows that $\|(C_{az+b} - C_{az+b'})(I-K_N)\| \leq \sup\limits_{n\geq N+1} \dfrac{\beta_{m_n}}{\beta_n}\bigabs{e^{-\lambda_n b} - e^{-\lambda_n b'}}$. Taking limits as $N\to\infty$ gives $\|C_{az+b} - C_{az+b'}\|_e \leq \limsup\limits_{n\to\infty} \dfrac{\beta_{m_n}}{\beta_n}\bigabs{e^{-\lambda_n b} - e^{-\lambda_n b'}}$.
	
	On the other hand, consider the probe functions $q_k$. Let $K$ be a compact operator on $\Hblam$. As before, we have $\|q_k\| = 1$ for all $k$ and $\|Kq_k\|\to 0$. We have
	\begin{align*}
	\|(C_{az+b} - C_{az+b'}) - K\| &\geq \limsup_{k\to\infty} \bigbracket{\|(C_{az+b} - C_{az+b'}) q_k\| - \|Kq_k\|}\\
	&\geq \limsup_{k\to\infty} \|(C_{az+b} - C_{az+b'}) q_k\|\\
	&= \limsup_{k\to\infty} \dfrac{\beta_{m_k}}{\beta_k}\bigabs{e^{-\lambda_k b} - e^{-\lambda_k b'}}.
	\end{align*}
	Taking infimum over all compact operators $K$ gives 
\[
\|C_{az+b} - C_{az+b'}\|_e \geq \limsup\limits_{n\to\infty} \dfrac{\beta_{m_n}}{\beta_n}\bigabs{e^{-\lambda_n b} - e^{-\lambda_n b'}}.
\]
\end{proof}

\begin{thm}\label{compactdiffcharacterisation}
Let $\beta_*\neq\pm\infty$, and let $C_{\vphi_1}$ and $C_{\vphi_2}$ be bounded composition operators on $\Hblam$. Then the difference $C_{\vphi_1} - C_{\vphi_2}$ is compact if and only if

(1) both $C_{\vphi_1}$ and $C_{\vphi_2}$ are compact\

or\

(2) $a\geq 1$, $\vphi_1(z) = az+b$ and $\vphi_2(z) = az+b'$, where
		\begin{enumerate}
			\item[(i)] we have
			\[
\limsup_{n\to \infty}r_n(a,\Re(b))>0\quad \mbox{and} \quad \limsup_{n\to \infty}r_n(a,\Re(b'))>0,
\] 			\item[(ii)] and  $\lim\limits_{k\to\infty} \dfrac{\beta_{m_k}}{\beta_k}\bigabs{e^{-\lambda_k b} - e^{-\lambda_k b'}} = 0$.
		\end{enumerate}

\end{thm}

\begin{proof}
Assume first that $C_{\vphi_1} - C_{\vphi_2}$ is compact but one of the operators $C_{\vphi_1}$ and $C_{\vphi_2}$ is not compact. Since the set of compact operators is a vector space, it implies that indeed both operators $C_{\vphi_1}$ and $C_{\vphi_2}$ are not compact. Now, according to Theorem \ref{ess-norm}, it means that condition (i) is satisfied (with the second inequality in $a'$ instead of $a$). Then, we can apply Lemma~\ref{compactdiffaffineprop} to get that $a = a'$, i.e. $\vphi_1(z)=az+b$ and $\vphi_2(z)=az+b'$. Condition (ii) now follows immediately from Proposition~\ref{prop:essentiel-norme-difference}.

Conversely, if both operators $C_{\vphi_1}$ and $C_{\vphi_2}$ are compact, then their difference is compact. Suppose now that $\vphi_1(z) = az+b$ and $\vphi_2(z) = az+b'$, and (i) and (ii) are satisfied. Proposition~\ref{prop:essentiel-norme-difference} implies that $\|C_{\vphi_1}-C_{\vphi2}\|_e=0$, which gives that  $C_{\vphi_1} - C_{\vphi_2}$ is compact.
\end{proof}

\begin{cor}\label{ehhhhhhhhhhhhhh} Let $\beta_*\neq\pm\infty$, and let $C_{z+ci}$ and $C_{z+c'i}$ ($c,c'\in\R$) be bounded composition operators on $\mathcal H(\beta,\Lambda)$. The operator $C_{z+ci} - C_{z+c'i}$ is compact if and only if $\lim\limits_{k\to\infty} \cos(\lambda_k(c - c')) = 1$.
\end{cor}

\begin{proof}
We see that $m_k=k$ and so $\beta_{m_k}=\beta_k$, $k\in\N$. Moreover, since $\Re(ci)=\Re(c'i)=0$, we have $r_n(a,\Re(ci))=r_n(a',\Re(c'i))=1$. Then, according to Proposition~\ref{cpt-a1} and Theorem~\ref{compactdiffcharacterisation}, we see that $C_{z+ci} - C_{z+c'i}$ is compact if and only if
\[
\lim_{k\to \infty}|e^{-\lambda_k ci}-e^{-\lambda_k c'i}|=0.
\]
An easy computation shows that
\[
|e^{-\lambda_k ci}-e^{-\lambda_k c'i}|^2=2-2\cos(\lambda_k(c-c')),
\]
which gives the result.
\end{proof}
Corollary~\ref{ehhhhhhhhhhhhhh} is an analogue of \cite[Theorem 4.12]{HHK} corresponding to the case when $\beta_*=\infty$.

\section{\bf Closed range and cyclicity}

In this section, we assume that $C_{az+b}$ is a bounded composition operator on $\Hblam$, which means that $a$ and $b$ satisfies conditions of Theorem~\ref{cri-bdd}.
\subsection{Closed range}

We denote by $R(C_\vphi)$ the range of $C_\vphi$, i.e. $C_\vphi(\Hblam)$. In this section we determine when $R(C_\vphi)$ is closed. Note that when $\varphi(z)=b$ (and $\lambda_1=0$), then $R(C_\vphi)$ is a one dimensional space (generated by the constant function $1=e^{-\lambda_1 z}$) and so it is closed.

\begin{prop}\label{prop:closeness}
Let $\beta_*\neq\pm\infty$, $a\geq 1$. Then, $R(C_{az+b})$ is closed if and only if $\inf\limits_{n\in\N} r_n(a,\Re(b)) > 0$.
\end{prop}

\begin{proof}
First note that it follows from the open mapping theorem and uniqueness principle for analytic functions that, since $\varphi$ is a non-constant analytic function, then $C_\vphi$ is injective. Now, suppose $B := \inf\limits_{n\in\N} r_n(a,\Re(b)) > 0$. Let $f(z) = \dsum_{n=1}^\infty a_n e^{-\lam_n z}$. Then,
	\begin{align*}
	\frac{\|C_{az+b} f\|^2}{\|f\|^2} = \frac{\dsum_{n=1}^\infty |a_n|^2\beta_{m_n}^2 e^{-2\lam_n \Re(b)}}{\dsum_{n=1}^\infty |a_n|^2\beta_n^2} = \frac{\dsum_{n=1}^\infty |a_n|^2\beta_n^2r_n(a,\Re(b))^2}{\dsum_{n=1}^\infty |a_n|^2\beta_n^2} \geq B^2.
	\end{align*}
	Thus $C_{az+b}$ is bounded from below and $R(C_{az+b})$ is hence closed.
	
	On the other hand, suppose $B := \inf\limits_{n\in\N} r_n(a,\Re(b)) = 0$. Let $(n_k)$ be a subsequence of $\N$ such that $r_{n_k}(a,\Re(b)) \to 0$. For each probe function $q_{n_k}$, we then have
	\begin{align*}
	\frac{\|C_{az+b} q_{n_k}\|^2}{\|q_{n_k}\|^2} = \frac{\frac{1}{\beta_{n_k}^2}\beta_{m_{n_k}}^2 e^{-2\lam_n \Re(b)}}{1} = r_{n_k}(a,\Re(b))^2 \to 0.
	\end{align*}
	It follows that $C_{az+b}$ cannot be bounded from below, and so $R(C_{az+b})$ is not closed.
\end{proof}
Proposition~\ref{prop:closeness} is an analogue of \cite[Theorem 6]{Doan-Khoi}.

\begin{rem}
	In the non-constant case, a compact $C_{az+b}$ cannot have closed range, and vice-versa if $C_{az+b}$ has closed range, it is not compact.
\end{rem}

\subsection{Cyclicity}

Let $\calH$ be a Hilbert space and $T:\calH\to \calH$ be a bounded operator. We define the orbit of a vector $x\in \calH$ (w.r.t. $T$) as the set
\begin{align*}
\Orb(T,x) = \bigcurlybracket{T^nx: n\in\N}.
\end{align*}
Furthermore, we recall that $T$ is said to be 
\begin{itemize}
	\item cyclic if there exists $x\in \calH$ such that
	\[
	\ol{\Span(\Orb(T,x))} = \calH,
	\]
	\item supercyclic if there exists $x\in \calH$ such that
	\[
	\ol{\bigcurlybracket{\mu y: y\in \Orb(T,x), \mu\in\C}} = \calH.
	\]
\end{itemize}

Note that for a given operator on an Hilbert space $\calH$,  if $T$ is supercyclic, then it is of course cyclic. We will discuss in this section the cyclicity and supercyclicity of the operators $C_\varphi$ on $\Hblam$. As we will see, $C_\varphi$ is never supercyclic but cyclicity will depend on the arithmetic properties of $(\lambda_n)_{n\geq 1}$.

It is trivial that if $\lambda_1=0$, then of course the operators $C_b$ and $C_z$ are not cyclic. Indeed, in both cases the orbit of $f$ (for every $f\in\Hblam$) contains only one function and so the orbit cannot generate a dense subspace.

We now split our study in two cases, depending whether $a=1$ or $a>1$.

\noindent \textbf{- The case $a=1$}.

To study this case, we need the following two general results. The first one is quite classical and can be found for instance in \cite[Chap. 18]{Halmos} for the finite dimensional case and in \cite[Lemma 1]{Seubert} for the general case. The proof of the second one can be found in \cite{Hilden} and uses the spectral mapping theorem.

\begin{lem}\label{Halmos}
Let $D$ be a diagonal operator on an Hilbert space $\mathcal H$, given by $De_n=s_n e_n$, $n\geq 1$, where $(e_n)_{n\geq 1}$ is an orthonormal basis of $\mathcal H$. Then $D$ is cyclic if and only if $s_n\neq s_m$, $n\neq m$.
\end{lem}
\begin{lem}\label{Hilden}
Let $T$ be a normal operator on a Hilbert space $\mathcal H$ of dimension greater than $1$. Then $T$ is not supercyclic.
\end{lem}

\begin{prop}\label{cyclicity-and-supercyclicity}
Let $\beta_*\neq\pm\infty$, $b\neq 0$. Then,
\begin{enumerate}
\item  $C_{z+b}$ is cyclic if and only if $(\lambda_n-\lambda_m)b\in\mathbb C\setminus 2\pi\mathbb Z$ whenever $n\neq m$, and
\item $C_{z+b}$ is not supercyclic.
\end{enumerate}
\end{prop}

\begin{proof}
(1)  Recall that the set of probe functions $q_k(z) = \dfrac{1}{\beta_k}e^{-\lam_k z}$ form a basis for $\Hblam$. Note also that for all $k\in\N$,
		\begin{align*}
		(C_{z+b} q_k)(z) = e^{-\lam_k b}q_k(z).
		\end{align*}
That means that $C_{z+b}$ is a diagonal operator with eigenvalue corresponding to $e^{-\lambda_k b}$. Thus (1) follows immediately from Lemma~\ref{Halmos}.

(2) Since $C_{z+b}$ is diagonal, it is in particular normal. So (2) follows from Lemma~\ref{Hilden}.
\end{proof}

\noindent \textbf{- The case $a>1$}.

The situation in this case is more interesting and the behavior of the iterates of $C_{az+b}$ will depend on the following notion.

Given $a\in R(\Lambda)$, $\Lambda=(\lam_n)_{n\geq 1}$, we define an \emph{initial point with respect to $a$} to be a term $\lam_k$ such that no $n < k$ exists such that $a^s\lam_n = \lam_k$ for some $s\in\N$. This is equivalent to say that, for $s\in\mathbb N_0$,
\[
a^s\lambda_n=\lambda_k\Longrightarrow s=0 \mbox{ and }n=k.
\]
Note that if $\lambda_1=0$, then $\lambda_2$ is necessarily an initial point. If $\lambda_1\neq 0$, then $\lambda_1$ is necessarily an initial point. Thus, in all cases, there exists at least one non zero initial point. Furthermore, $(\lam_n)_{n\geq 1}$ has only one non-zero initial point w.r.t. $a$ if and only if non-zero terms of $(\lam_n)_{n\geq 1}$ are in geometric progression with common ratio $a$.

\begin{prop}\label{Ireallydunnowhattouseforlabelsalready}
	Let $\beta_*\neq\pm\infty$, $a>1$, and suppose $(\lam_n)_{n\geq 1}$ has precisely one initial point with respect to $a$. Then, $C_{az+b}$ is cyclic but not supercyclic.
\end{prop}

\begin{proof}
Since $(\lam_n)_{n\geq 1}$ has precisely one initial point, it must be the case that $\lam_n = \lam_1 a^{n-1}$ for every $n$, and $\lam_1 \neq 0$. Let $f(z) = e^{-\lam_1 z}$. By induction, we easily check that
	\begin{align*}
	(C_{az+b}^k f)(z) = \exp (-(\lam_1 + \lam_2 + \dots + \lam_k)b) e^{-\lam_{k+1} z}.
	\end{align*}
In particular, $\Orb(C_{az+b},f)$ contains the vectors of the basis $(q_k)_{k\geq 1}$ of $\Hblam$. Therefore $C_{az+b}$ is cyclic.
	
	To show $C_{az+b}$ is not supercyclic, we show that for every $f\in\Hblam$ we can find a $g$ such that $\dist(g, \bigcurlybracket{\mu y: y\in\Orb (C_{az+b}, f), \mu\in\C})$ is bounded below by some non-zero constant. There are two cases.
	\begin{itemize}
		\item Case 1: $f(z) = a_1e^{-\lam_1 z}$ for some $0\neq a_1\in\C$. In this case, consider the function $g(z) = e^{-\lam_1 z} + e^{-\lam_2 z}$. It is then easy to see that
		\begin{align*}
		\|\mu f - g\|^2 &= |\mu a_1 - 1|^2 \beta_1^2 + \beta_2^2 \geq \beta_2^2,\\
		\|\mu C_{az+b} f - g\|^2 &= \beta_1^2 + |\mu a_1 e^{-\lam_1 b} - 1|^2 \beta_2^2 \geq \beta_1^2,\\
		\|\mu C_{az+b}^k f - g\|^2 &\geq \beta_1^2 + \beta_2^2,\qquad \forall k\geq 2.
		\end{align*}
		Therefore $\dist(g, \bigcurlybracket{\mu y: y\in\Orb (C_{az+b}, f), \mu\in\C}) \geq \min\bigcurlybracket{\beta_1, \beta_2}$.

		\item Case 2: $f(z) \neq a_1e^{-\lam_1 z}$ for any $a_1\in\C$.
		
		Write $f(z) = \dsum_{n=1}^\infty a_n e^{-\lam_n z}$. Choose $\kappa \neq 0, -a_1$ and pick $g(z) = \kappa e^{-\lam_1 z}$. Note that for all $k\geq 1$ we have that the coefficient of $e^{-\lam_1 z}$ in the representation of $C_{az+b}^k f$ is $0$, it follows then that for each $k\geq 1$, we have $\|\mu C_{az+b}^k f - g\|^2 \geq |\kappa|^2\beta_1^2$.
		
		It remains to consider $k=0$. We have
		\begin{align*}
		(\mu f - g)(z) &= \mu \dsum_{n\geq 1} a_n e^{-\lam_n z} - \kappa e^{-\lam_1 z}\\
		&= (\mu a_1 - \kappa) e^{-\lam_1 z} + \dsum_{n\geq 2} \mu a_n e^{-\lam_n z}.
		\end{align*}
		Since $f(z) \neq a_1 e^{-\lam_1 z}$ for any $a_1$, there exists $k\geq 2$ such that $a_k \neq 0$. We have
		\begin{align*}
		\|\mu f - g\|^2 &\geq |\mu a_1 - \kappa|^2\beta_1^2 + |\mu a_k|^2 \beta_k^2.
		\end{align*}
		By minimizing over $\Re(\mu)$ and noting that the minimum remains unchanged if we replaced $\kappa$ and $a_1$ with their moduli, the RHS satisfies
		\begin{align*}
		|\mu a_1 - \kappa|^2\beta_1^2 + |\mu a_k|^2 \beta_k^2 \geq \dfrac{|\kappa a_k|^2 \beta_1^2\beta_k^2}{|a_1|^2\beta_1^2+|a_k|^2\beta_k^2} =: D,
		\end{align*}
		where the inequality follows because the expression $ax^2 + b(x-c)^2$ where $a,b,c\in\R$, $a, b > 0$, has minimum $\frac{abc^2}{a+b}$ as $x$ ranges over $\R$.
		
		Thus $\dist(g, \bigcurlybracket{\mu y: y\in\Orb (C_\vphi, f), \mu\in\C}) \geq \min\{|\kappa|\beta_1, \sqrt{D}\}$.
	\end{itemize}
\end{proof}

\begin{prop}\label{ughhhh}
	Let $\beta_*\neq\pm\infty$, $a>1$, and suppose $(\lam_n)_{n\geq 1}$ has at least two non-zero initial points with respect to $a$. Then, $C_{az+b}$ is not cyclic.
\end{prop}

\begin{proof}
	Let $\lam_p, \lam_q$ denote two non-zero initial points. Suppose for the sake of contradiction that $C_{az+b}$ is cyclic. Let $f(z) = \dsum_{n=1}^\infty a_ne^{-\lam_n z}$ be a cyclic vector for $C_{az+b}$.
	
	We claim that $a_p, a_q\neq 0$. Assume to the contrary that one of them is zero, say WLOG $a_p = 0$. Recall that $\Orb(C_{az+b},f) = \bigcurlybracket{f, C_{az+b} f, C_{az+b}^2 f, \dots}$. Note that
	\begin{align*}
	(C_{az+b}^k f)(z) = \dsum_{n=1}^\infty a_n \exp\bigbracket{-b\frac{1-a^k}{1-a} }\exp\bigbracket{-a^k \lam_{n} z}.
	\end{align*}
	Since $\lam_p$ is a non-zero initial point, $a^k\lambda_n\neq \lambda_p$ for all $k,n\geq 1$. It follows that the coefficient of $\lam_p$ in any function in $\Orb(C_{az+b},f)$ is always $0$. Hence for any $f_1\in \Span(\Orb(C_{az+b},f))$, one has
	\begin{align*}
	\|f_1-q_p\|^2 \geq 1.
	\end{align*}
	Therefore $\Span(\Orb(C_{az+b},f))$ cannot be dense in $\Hblam$, a contradiction, completing the proof of the claim.
	
	Hence $a_p, a_q\neq 0$. Since $a^k\lambda_n\neq \lambda_p,\lambda_q$ for all $k,n\geq 1$, the only function in $\Orb(C_{az+b},f)$ with non-zero coefficients for $e^{-\lam_p z}$ and $e^{-\lam_q z}$ terms in $\Span(\Orb(C_{az+b},f))$ is $f$ itself.
	
	Consider now the function $g(z) = a_p e^{-\lam_p z} + 2a_q e^{-\lam_q z} \in\Hblam$. Let $F\in \Span(\Orb(C_{az+b},f))$ and fix a representation of $F$ in elements of $\Orb(C_{az+b},f)$. Let $w$ the coefficient of $f$ in this representation. Then we have
	\begin{align*}
	\|F-g\|^2 &\geq |w-1|^2|a_p|^2\beta_p^2 + |w-2|^2|a_q|^2\beta_q^2.
	\end{align*}
	Let $\xi = |a_p|^2\beta_p^2$ and $\eta = |a_q|^2\beta_q^2$. Write $w = w'+1$. Then,
	\begin{align*}
	\|F-g\|^2 &\geq \xi|w'|^2 + \eta|w'-1|^2\\
	&\geq \xi\Re(w')^2 + \eta(\Re(w') - 1)^2 \geq \dfrac{\xi\eta}{\xi + \eta},
	\end{align*}
	where the last inequality is minimized in the same manner as in Proposition~\ref{Ireallydunnowhattouseforlabelsalready}.
	
	It follows that $\Span(\Orb(C_{az+b},f))$ cannot be dense in $\Hblam$. Contradiction.
\end{proof}

It remains to study the case when $(\lam_n)_{n\geq 1}$ has precisely one zero and one non-zero initial point with respect to $a$. In other words, $\lambda_1=0$ and $\lambda_n=\lambda_2 a^{n-2}$, $n\geq 2$. In that case, we will give a sufficient condition to ensure that $C_{az+b}$ is cyclic. Since cyclicity is invariant under unitary transformation, we will work for convenience with a transcription of our problem in the Hardy space of the unit disc $H^2$. Recall that
\begin{align*}
H^2 = \bigcurlybracket{f(z) = \dsum_{k=0}^\infty a_k z^k \in\Hol(\D): f(z) = \dsum_{k=0}^\infty |a_k|^2 < \infty}.
\end{align*}
The Hardy space has the canonical orthonormal basis $\{1,z,z^2,\dots\}$ and for the remainder of this section we will denote by $e_n$ the vectors of this basis, that is $e_n(z) = z^n$ for all $n\geq 0$.

\begin{prop}
	Let $\beta_*\neq\pm\infty$, $a>1$, and suppose $(\lam_n)$ has precisely one zero and one non-zero initial point with respect to $a$. Then the matrix $M$ of $C_{az+b}$ with respect to the orthonormal basis $(q_k)$ is
	\begin{align*}
	M = \bigbracket{\begin{matrix}
		1 & 0 & 0 & \dots\\
		0 & 0 & 0 & \dots\\
		0 & r_2(a,b) & 0 & \dots\\
		0 & 0 & r_3(a,b) & \dots\\
		\vdots & \vdots & \vdots & \ddots
		\end{matrix}}.
	\end{align*}
\end{prop}

Now, it is clear that our problem is a particular case of the following more general problem.
\begin{ques}
	Let $\alpha = (\alpha_n)$ be a bounded sequence of complex numbers and suppose that $T:H^2\to H^2$ is the linear map whose matrix, with respect to the orthonormal basis $\{e_0,e_1,e_2,\dots\}$, is given by
	\begin{align*}
	\bigbracket{\begin{matrix}
		1 & 0 & 0 & \dots\\
		0 & 0 & 0 & \dots\\
		0 & \alpha_1 & 0 & \dots\\
		0 & 0 & \alpha_2 & \dots\\
		\vdots & \vdots & \vdots & \ddots
		\end{matrix}}
	\end{align*}
Is $T$ cyclic?
\end{ques}

Note that since $(\alpha_n)$ is supposed to be bounded, it is easy to see that $T$ defines a bounded operator on $H^2$. It turns out that $T$ is linked with a weighted shift operator. In the following, we denote by $H_0^2 = \ol{\Span\{e_1,e_2,\dots\}}$.

\begin{defn}
	Let $\alpha = (\alpha_n)_{n=1}^\infty$ be a bounded sequence. We define the weighted forward shift operator $S_\alpha$ on $H^2_0$ by $S_\alpha e_n = \alpha_n e_{n+1}$, $n\geq 1$.
\end{defn}
Note that we have $T=I\oplus S_\alpha: H^2=\mathbb Ce_0\oplus H_0^2\to H^2=\mathbb Ce_0\oplus H_0^2$. In other words,
\[
T(\gamma e_0\oplus x)=\gamma e_0\oplus S_\alpha x\qquad (\gamma,x)\in \mathbb C\times H_0^2.
\]
By induction and basic computation we have the following two lemmas.

\begin{lem}\label{polylemoops}
	Let $p\in\C[X]$. Then, $(p(T))(\gamma e_0 \oplus x) = \gamma p(1)e_0 \oplus p(S_\alpha)x$.
	\end{lem}

\begin{lem}\label{lem=psalpha}
	Let $p(z) = \dsum_{j=0}^{d} a_j z^j$ be a polynomial. Let $\omega_j := \alpha_1\alpha_2\dots \alpha_j$ for all $j\geq 1$. Then,
	\begin{align*}
	p(S_\alpha) e_1 = a_0e_1 + \dsum_{j=1}^{d} a_j \omega_j e_{j+1}.
	\end{align*}
\end{lem}

The next lemma provides an equivalence we use in our investigation.

\begin{lem}\label{kappapielemma}
	The following are equivalent.
	\begin{enumerate}
		\item $T$ is cyclic.
		\item There exists $\kappa \in\C$ and $x\in H_0^2$ such that for all $\eps > 0$, $\mu\in\C$, and $y\in H_0^2$, there exists $p\in\C[X]$ such that
		\begin{align*}
		|\kappa p(1) - \mu| < \eps \qquad \text{and}\qquad \|p(S_\alpha)x - y\| < \eps.
		\end{align*}
	\end{enumerate}
\end{lem}

\begin{proof}
	By definition, $T$ is cyclic if and only if there exists $\kappa \in\C$ and $x\in H_0^2$ such that for all $\eps > 0$, $\mu\in\C$, and $y\in H_0^2$, there exists $p\in\C[X]$ such that
	\begin{align*}
	\|p(T)(\kappa e_0 \oplus x) - (\mu e_0 \oplus y)\|^2 < \eps.
	\end{align*}
	
	By Lemma \ref{polylemoops} and by virtue of being direct sums, we have
	\begin{align*}
	\|p(T)(\kappa e_0 \oplus x) - (\mu e_0 \oplus y)\|^2 &= \|(\kappa p(1) e_0 \oplus p(S_\alpha)x) - (\mu e_0 \oplus y)\|^2\\
	&= \|(\kappa p(1) - \mu)e_0 \oplus (p(S_\alpha)x - y)\|^2\\
	&= |\kappa p(1) - \mu|^2 + \|p(S_\alpha)x - y\|^2.
	\end{align*}
	Hence the condition in the definition can be rewritten as
	\begin{align*}
	|\kappa p(1) - \mu|^2 + \|p(S_\alpha)x - y\|^2 < \eps.
	\end{align*}
	Since $\eps > 0$ is an arbitrary variable, this completes the proof.
\end{proof}

With this lemma, the question becomes one of whether such a polynomial $p\in\C[X]$ can always be chosen so that the second statement of Lemma \ref{kappapielemma} holds. To that end, we have the following result.

\begin{prop}\label{eeeeesh}
Let $(\alpha_j)$ be a bounded sequence of complex numbers, let $\omega_j = \alpha_1\alpha_2\dots\alpha_j$ for each $j\geq 1$, and suppose that
\[
\sum_{j=1}^\infty \frac{1}{|w_j|^2}=\infty.
\]
Let $\nu\in\C$ be given.
Then, there exists a sequence of non-constant polynomials $(u_k)\in\C[X]$ such that
	\begin{align*}
	u_k(1) = \nu \qquad\text{and}\qquad \|u_k(S_\alpha)e_1\|\to 0.
	\end{align*}
\end{prop}

\begin{proof}
Without loss of generality, we can assume that $\nu$ is a non-negative real number.
	For a general $u_k\in\C[X]$ we write $u_k(z) = \dsum_{j=0}^{d_k} b_{j,k} z^j$. The question is to define appropriate choices of $b_{j,k}$ that fulfill the conditions of the lemma.
	
	We find $u_k(z)$ with the form $u_k(z) = \nu + (z-1)p_k(z)$, where the $p_k$'s are polynomials. Note that such a $u_k$ will always have $u_k(1) = \nu$. Write $p_k(z) = \dsum_{j=0}^{s_k} a_{j,k} z^j$. Then we have
	\begin{align*}
	u_k(z) &= \nu + \dsum_{j=0}^{s_k} a_{j,k} z^{j+1} - \dsum_{j=0}^{s_k}a_{j,k} z^j\\
	&= (\nu - a_{0,k}) + \dsum_{j=1}^{s_k} (a_{j-1,k} - a_{j,k}) z^j + a_{s_k, k} z^{s_k + 1}.
	\end{align*}
	This gives by Lemma~\ref{lem=psalpha}
	\begin{align*}
	u_k(S_\alpha)e_1 = (\nu - a_{0,k})e_1 + \dsum_{j=1}^{s_k} (a_{j-1,k} - a_{j,k}) \omega_j e_{j+1} + a_{s_k, k} \omega_{s_k + 1} e_{s_k+2},
	\end{align*}
	and so
	\begin{align*}
	\|u_k(S_\alpha)e_1\|^2 = |\nu - a_{0,k}|^2 + \dsum_{j=1}^{s_k} |a_{j-1,k} - a_{j,k}|^2|\omega_j|^2 + |a_{s_k, k}|^2|\omega_{s_k+1}|^2.
	\end{align*}
Since $\dsum_{n=1}^{\infty}|w_j|^{-2}=\infty$, it is well known that we can find a sequence of positive real numbers $(h_j)_{j\geq 1}$ such that
\[
\sum_{n=1}^\infty \frac{1}{h_j|w_j|}=\infty\quad\mbox{and}\quad \sum_{n=1}^\infty \frac{1}{h_j^2}<\infty.
\]

	Assume $\nu>0$. Pick $a_{0,k} = \nu$. Define $\eps_k>0$ such that
	\begin{align*}
	\eps_k \bigbracket{\dfrac{1}{h_1|\omega_1|} + \dfrac{1}{h_2|\omega_2|} + \dots + \dfrac{1}{h_{s_k}|\omega_{s_k}|}} = \nu,
	\end{align*}
	and define recursively
	\begin{align*}
	a_{j,k} = a_{j-1,k} - \dfrac{\eps_k}{h_j|\omega_j|},\qquad 1\leq k\leq s_k.
	\end{align*}
	
	It is easily checked that $a_{s_k,k} = 0$ and $|a_{j-1,k} - a_{j,k}|^2|\omega_j|^2 = \dfrac{|\eps_k|^2}{h_j^2}$ for all $1\leq j\leq s_k$. Under this choice of coefficients, we therefore have
	\begin{align*}
	\|u_k(S_\alpha)e_1\|^2 = \dsum_{j=1}^{s_k} \frac{|\eps_k|^2}{h_j^2} < |\eps_k|^2\sum_{j=1}^\infty \frac{1}{h_j^2} =: |\eps_k|^2 M.
	\end{align*}
	Since $M$ is finite, we are done if we can pick values of $s_k$ so that $\eps_k$ goes to zero. But by assumption, $\dsum_{j=1}^{s_k} \dfrac{1}{h_j|\omega_j|}$ will grow arbitrarily large, so indeed $s_k$ can be chosen to make $\eps_k$ arbitrarily small.
	
	If $\nu = 0$, we fix $\eps_k > 0$, pick $a_{0,k} = \eps_k$, and define $a_{j,k}$ recursively as before, and we then have
	\begin{align*}
	a_{s_k,k} = \eps_k\bigbracket{1 - \dsum_{j=1}^{s_k} \frac{1}{h_j|\omega_j|}}.
	\end{align*}
	It then follows that
	\begin{align*}
	\|u_k(S_\alpha)e_1\|^2 &= |0 - a_{0,k}|^2 + \dsum_{j=1}^{s_k} |a_{j-1,k} - a_{j,k}|^2|\omega_j|^2 + |a_{s_k, k}|^2|\omega_{s_k+1}|^2\\
	&= \eps_k^2 C_k,
	\end{align*}
	where $C_k$ is a constant value dependent only on the choice of value of $s_k$. It follows that one may simply pick $s_k$ freely and choose $\eps_k$ arbitrarily small to make $\|u_k(S_\alpha)e_1\|$ arbitrarily small.
\end{proof}

We are now ready to prove the following.

\begin{thm}\label{woahtheorem}
	Suppose $(|\omega_j|^{-1})\notin\ell^2$. Then, $T$ is cyclic.
\end{thm}

\begin{proof}
	Pick $\kappa = 1$ and $x = e_1$ in Statement (2) of Lemma \ref{kappapielemma}. Fix $\eps > 0$, $\mu \in\C$, and $y\in H_0^2$. We need to show that $p\in\C[X]$ exists such that
	\begin{align*}
	|p(1) - \mu| < \eps\qquad\text{and}\qquad \|p(S_\alpha)e_1 - y\| < \eps.
	\end{align*}
	
	Write $y = \dsum_{n=1}^\infty b_n z^n$. Since $y\in H_0^2$, there exists $N \geq 1$ such that $\dsum_{n=N+1}^\infty |b_n|^2 < \dfrac{\eps}{2}$. Fix this $N$.
	
	We claim we can choose a polynomial $\wt{p}$ such that $\|\wt{p}(S_\alpha)e_1 - y\| < \dfrac{\eps}{4}$. Write $\wt{p}(z) = \dsum_{n=0}^{N-1} c_n z^n$ and let $y'(z) = \dsum_{n=1}^N b_n z^n$. The question is to show that there exists a choice of coefficients $c_n$. We have
	\begin{align*}
	(\wt{p}(S_\alpha)e_1 - y')(z) &= \dsum_{n=0}^N c_n S_\alpha^n (z) - \dsum_{n=1}^{N-1} b_n z^n\\
	&= c_0 z + \dsum_{n=1}^{N-1} c_n \omega_n z^{n+1} - \dsum_{n=1}^N b_n z^n\\
	&= (c_0 - b_1)z + \dsum_{n=2}^{N} (c_{n-1} \omega_{n-1} - b_n) z^n.
	\end{align*}
	Clearly then
	\begin{align*}
	\|\wt{p}(S_\alpha)e_1 - y'\|^2 = |c_0 - b_1|^2 + \dsum_{n=2}^{N}|c_{n-1} \omega_{n-1} - b_n|^2.
	\end{align*}
	Since the $\omega_n$'s and the $b_n$'s are known values with $\omega_n \neq 0$ for all $n$, it follows that we can pick an appropriate choice of $c_n$'s in every instance such that $\|\wt{p}(S_\alpha)e_1 - y\|^2 < \dfrac{\eps}{4}$, for instance
	\begin{align*}
	c_0 = \frac{\sqrt{\eps}}{2}+ b_1,\qquad\text{and}\qquad c_n = \frac{b_n}{\omega_{n-1}}, \quad\forall n\geq 2.
	\end{align*}
	
	Since $(|\omega_j|^{-1})\notin\ell^2$, by Proposition \ref{eeeeesh} there exists a sequence $(u_k)\subset \C[X]$ such that
	\begin{align*}
	u_k(1) = \mu - \wt{p}(1) \qquad\text{and}\qquad \|u_k(S_\alpha)e_1\|\to 0.
	\end{align*}
	In particular there exists $k$ large enough such that $\|u_k(S_\alpha)e_1\| < \dfrac{\eps}{4}$. Fix this $k$.
	
	Now, set $p = \wt{p} + u_k$. Then we have by the triangle inequality that
	\begin{align*}
	\|p(S_\alpha)e_1 - y\| &= \|\wt{p}(S_\alpha)e_1 + u_k(S_\alpha)e_1 - (y - y') - y'\|\\
	&\leq \|\wt{p}(S_\alpha)e_1 - y'\| + \|y - y'\| +  \|u_k(S_\alpha)e_1\| < \eps.
	\end{align*}
	Also, note that $p(1) = \wt{p}(1) + u_k(1) = \mu$, which gives
	\begin{align*}
	|p(1) - \mu| = 0 < \eps.
	\end{align*}
	Hence indeed we can find such a $p$ as per Statement (2) of Lemma \ref{kappapielemma}. By this same lemma, $T$ is cyclic. This proves the theorem.
\end{proof}

The above leads to a partial solution of the remaining case for cyclicity of $C_{az+b}$. 

\begin{cor}\label{cor:cyclicity}
Let $\beta_* \neq\pm\infty$, $a>1$, and suppose $(\lam_n)$ has precisely one zero and one non-zero initial point with respect to $a$.
Let $w_j=\displaystyle\prod_{k=2}^{j+1}r_k(a,b)$. Assume that $(w_j^{-1})_j\not\in\ell^2$. Then $C_{az+b}$ is cyclic.
\end{cor}

\begin{cor}\label{gaaaaaaaaaaaaaaaaaaaaaah}
	Let $\beta_* \neq\pm\infty$, $a>1$, and suppose $(\lam_n)$ has precisely one zero and one non-zero initial point with respect to $a$. Let $b\in\C$ and suppose that $\Re(b) > a\beta_*$ and $\Re(b) \geq 0$. Then, $C_{az+b}$ is cyclic.
\end{cor}

\begin{proof}
	Assume for the sake of contradiction that $C_{az+b}$ is not cyclic. Hence by Corollary~\ref{cor:cyclicity}, we have $(|\omega_j|^{-1})\in\ell^2$. In particular, $|\omega_j|^{-1} \to 0$. But a direct computation shows that
	\begin{align*}
	|\omega_j|^{-1} = \frac{\beta_2}{\beta_{j+2}} e^{\lam_2 \Re(b)(1+a+\dots+a^{j-1})},
	\end{align*}
and since $\Re(b)\geq 0$, we get $\dfrac{1}{\beta_{j+2}} e^{\lam_2 \Re(b) a^{j-1}} \to 0$.
	
	Recall now that $\liminf\limits_{n\to\infty} \dfrac{\log\beta_n}{\lam_n} =: \beta_*$ is finite. Therefore for all $\eps > 0$, there exists a strictly increasing subsequence $(n_p)$ of the natural numbers such that $\dfrac{\log\beta_{n_p}}{\lam_{n_p}} < \beta_* + \eps$. In particular we have
	\begin{align*}
	\frac{1}{\beta_{n_p}} > e^{-\lam_{n_p} (\beta_* + \eps)} = e^{-\lam_2 a^{n_p - 2} (\beta_* + \eps)}, \qquad \forall n_p \geq 2.
	\end{align*}
	
	Fix $\eps > 0$ small such that $\Re(b) > a(\beta_* + \eps)$. Choose the subsequence $(n_p)$ as in the above, and write $j_p = n_p - 2$ for all $n_p > 2$. Then in particular we have
	\begin{align*}
	\frac{1}{\beta_{j_p+2}} e^{\lam_2 \Re(b)a^{j_p-1}} &= \frac{1}{\beta_{n_p}} e^{\lam_2 \Re(b)a^{n_p-3}}\\
	&> e^{-\lam_2 a^{n_p - 2} (\beta_* + \eps)}e^{\lam_2 \Re(b) a^{n_p-3}}\\
	&= e^{\lam_2a^{n_p-3}(\Re(b) - a(\beta_* + \eps))}.
	\end{align*}
	By assumption, $\Re(b) > a(\beta_* + \eps)$, i.e. $\Re(b) - a(\beta_* + \eps) > 0$. Hence the RHS has infinite limit. But this means there exists a subsequence $(j_p)$ such that $|\omega_{j_p}|^{-1} \to \infty$. Contradiction.
\end{proof}

We end this section with a result which completes Proposition~\ref{cyclicity-and-supercyclicity} on supercyclicity. As the proof methodology is similar to arguments in Proposition~\ref{Ireallydunnowhattouseforlabelsalready}, we omit the full details and provide a sketch.

\begin{prop}
	Let $\beta_* = \pm\infty$, $a>1$, and suppose $(\lam_n)$ has precisely one zero and one non-zero initial point with respect to $a$. Then, $C_{az+b}$ is not supercyclic.
\end{prop}

\begin{proof}[Proof (Outline)]
	The idea is for every $f\in\Hblam$, we can find a function $g$ and a constant $B > 0$ such that $\dist(g, \bigcurlybracket{\mu y: y\in\Orb(C_{az+b}, f), \mu\in\C}) \geq B$. The proof is split into four cases.
	
	\textbf{- Case 1:} $f(z)\eq a_1$ $(a_1\in\C)$. The result is obvious as $\Orb(C_{az+b}, f)$ has only one element, so $\bigcurlybracket{\mu y: y\in\Orb(C_{az+b}, f), \mu\in\C}$ has dimension 1.
	
	\textbf{- Case 2:} $f(z) = a_2e^{-\lam_2 z}, a_2\neq 0$. Pick $g(z)\eq b > 0$. Since all $h\in \Orb(C_{az+b}, f)$ have non-zero constant part, it can be verified that taking $B = |b|\beta_1$ works.
	
	\textbf{- Case 3:} $f(z) = a_1 + a_2e^{-\lam_2 z}$ $(a_1,a_2\neq 0)$. Pick $g(z) = a_1 + (a_2+1)e^{-\lam_2 z}$. Since $\lam_2$ is a non-zero initial point, following the proof of Proposition~\ref{Ireallydunnowhattouseforlabelsalready} (Case 2) yields $\|\mu C_{az+b}^k f - g\|^2 \geq |a_2+1|^2\beta_2^2$ for all $k\geq 1$. Furthermore the computation for $k=0$ gives
	\begin{align*}
	\|\mu f - g\|^2 \geq |\mu - 1|^2|a_1|^2\beta_1^2 + |(\mu - 1)a_2 + 1|^2\beta_2^2,
	\end{align*}
	which is minimized in a similar manner as in Proposition~\ref{Ireallydunnowhattouseforlabelsalready}.
	
	\textbf{- Case 4:} $\exists k\geq 3$ such that $a_k \neq 0$. Pick $\kappa \neq 0, -a_2$ and $g(z) = \kappa e^{-\lam_2 z}$. This is again similar to Proposition~\ref{Ireallydunnowhattouseforlabelsalready}.
\end{proof}

\section{\bf Complex symmetry}

In this section we investigate the complex symmetry property of $C_\vphi$. In this section we adopt the convention that
\begin{equation}
e_n(z) = \beta_n q_n(z) = e^{-\lam_n z}.
\end{equation}

\subsection{Composition conjugations}

\begin{defn}
	Let $\calH$ be a $\C$-Hilbert space. A map $\calC: \calH \to \calH$ satisfying the conditions of
	\begin{enumerate}
		\item isometry: $\|\calC x\| = \|x\|, \forall x\in\calH$,
		\item involutivity: $\calC\calC = I$, the identity map, and
		\item anti-linearity: $\calC(\lambda x + \mu y) = \ol{\lambda} \calC x + \ol{\mu} \calC y, \forall x,y\in\calH, \forall \lambda,\mu\in\C$,
	\end{enumerate}
	is called a conjugation on $\calH$.
\end{defn}

\begin{defn}
	Let $\xi:\C_{L/2-\beta_*} \to \C_{L/2-\beta_*}$ be an analytic function. Define
	\begin{align*}
	J_\xi f(z) = \ol{f(\ol{\xi(z)})},\qquad \forall f\in\Hblam.
	\end{align*}
	If $J_\xi$ satisfies the definition of a conjugation, then it is called a composition conjugation.
\end{defn}

We will investigate composition conjugations induced by polynomial functions~$\xi$.
\begin{thm}\label{conjugationswoohoo!}
Assume that $\xi$ is a polynomial.	The following statements are equivalent.
	\begin{enumerate}
		\item $J_\xi$ is a conjugation on $\Hblam$.
		\item $\xi(z) = z + ci, c\in\R$.
	\end{enumerate}
\end{thm}

\begin{proof}
$(1) \Longrightarrow (2)$: first let us note that since $J_\xi$ is a conjugation on $\Hblam$, then $\xi$ induces a bounded composition operator on $\Hblam$. Indeed, let $f(z) = \dsum_{n=1}^\infty a_n e^{-\lambda_n z} \in \Hblam$. Then the function
\begin{align*}
\widetilde{f}(z) = \dsum_{n=1}^\infty \ol{a_n} e^{-\lambda_n z}.
\end{align*}
also belongs to $\Hblam$, and so $J_\xi \widetilde{f}\in\Hblam$. Note now that
\[
J_\xi \widetilde{f}(z)=\overline{\widetilde{f}(\overline{\xi(z)})}=\sum_{n=1}^\infty a_n e^{-\lambda_n \xi(z)}=(C_\xi f)(z).
\]
Therefore, we get $C_\xi f\in\Hblam$ as required. In particular, according to Theorem~\ref{cri-bdd}, $\xi$ should be of the form $\xi(z)=az+b$, where $a=0$ and $\Re(b)>\frac{L}{2}-\beta_*$ or $a\geq 1$ and $\Re(b)\geq (1-a)(\frac{L}{2}-\beta_*)$.

First, we will show that $a\neq 0$. Argue by absurdity and assume that $a=0$. Then $\xi(z)=b$ and $(J_\xi f)(z)=\overline{f(\overline{b})}$. Apply that equation to the vector of the orthonormal basis $q_k(z)=\frac{1}{\beta_k}e^{-\lambda_k z}$, which gives
\[
(J_\xi q_k)(z)=\frac{1}{\beta_k}e^{-\lambda_k z}.
\]
Since $J_\xi$ is an isometry, we get
\[
\frac{\beta_1}{\beta_k}e^{-\lambda_k \Re(b)}=1,
\]
that is
\[
-\Re(b)=\frac{\log(\beta_k)}{\lambda_k}-\frac{\log(\beta_1)}{\lambda_k},\qquad k\geq 2.
\]
Taking the limit when $k\to \infty$, we get $\Re(b)=-\beta_*$. In particular, $\beta_*\neq \pm\infty$ and the condition $\Re(b)>\frac{L}{2}-\beta_*$ gives the contradiction. Hence $a\neq 0$.

Let us show now that $a=1$ and $b\in i\mathbb R$. Remind that $e_n(z)=e^{-\lambda_n z}\in\Hblam$. Since $J_\xi$ is conjugation on $\Hblam$, we have $e_n=J_\xi^2 e_n$, which gives
\begin{align*}
e^{-\lambda_n z}=& J_\xi(\overline{e_n)(\overline{\xi(z)})})\\
=&J_\xi(e^{-\overline{\lambda_n}\xi(z)})\\
=&e^{-\lambda_n \overline{\xi(\overline{\xi(z)})}}.
\end{align*}
By analyticity, we get that for every $n\geq 1$, there exists $k_n\in\mathbb Z$ such that
\[
-\lambda_n \overline{\xi(\overline{\xi(z)})}=-\lambda_n z+2i\pi k_n.
\]
In particular, since $\lambda_2>\lambda_1\geq 0$, we deduce that
\[
\overline{\xi(\overline{\xi(z)})}=z-\frac{2i\pi k_2}{\lambda_2}=z+ic_1,
\]
where $c_1=-2\pi \frac{k_2}{\lambda_2}\in\mathbb R$. Recall now that $\xi(z)=az+b$. An easy computation shows that $\overline{\xi(\overline{\xi(z)})}=a^2 z+\bar a b+\bar b$. Hence, we see that $a^2=1$, that $a=1$ (because $a\geq 0$) and $b+\bar b=ic_1$, that is $2\Re(b)=ic_1$. This implies that $\Re(b)=c_1=0$. Finally, we get that $\xi(z)=z+i\Im(b)=z+ic$, with $c\in\R$.

$(2) \Longrightarrow (1)$: assume that $\xi(z)=z+ic$, $c\in\R$. Since $\Re(\xi(z))=\Re(z)$, it is clear that $\xi$ maps $\C_{\frac{L}{2}-\beta_*}$ into itself. Then $J_\xi$ is well defined. Now, if $f(z)=\displaystyle\sum_{n=1}^\infty a_n e^{-\lambda_n z}\in\Hblam$, then
\[
(J_\xi f)(z)=\sum_{n=1}^\infty \overline{a_n}e^{-\lambda_n \xi(z)}=\sum_{n=1}^\infty \overline{a_n}e^{-ic\lambda_n}e^{-\lambda_n z}.
\]
Since $\left|\overline{a_n} e^{-ic\lambda_n}\right|=|a_n|$, we see that $J_\xi f\in\Hblam$. Moreover,
\[
\|J_\xi f\|^2=\sum_{n=1}^\infty \beta_n^2|a_n|^2=\|f\|^2.
\]
Hence, $J_\xi$ is an isometry from $\Hblam$ into itself. It is of course antilinear. It remains to check that $J_\xi$ is involutive. But, note that
\[
(J_\xi e_n)(z)=e^{-\lambda_n \xi(z)}=e^{-ic\lambda_n}e^{-\lambda_n z}=e^{-ic\lambda_n}e_n(z),
\]
and so
\[
J_\xi^2(e_n)=J_\xi(e^{-ic\lambda_n}e_n)=e^{ic\lambda_n}J_\xi(e_n)=e^{ic\lambda_n}e^{-ic\lambda_n} e_n=e_n
\]
By linearity and continuity of $J_\xi^2$, and density of $\bigvee(e_n:n\geq 1)$ in $\Hblam$, we get that $J_\xi^2=I$. That proves that $J_\xi$ is a conjugation on $\Hblam$.
\end{proof}

\subsection{Complex symmetry}

Having proven the form of composition conjugations $J_\xi$ on $\Hblam$, we can now consider the complex symmetry property of bounded composition operators on $\Hblam$ with respect to conjugations $J_\xi$.

\begin{defn}
	Let $T$ be a continuous linear operator mapping a Hilbert space $\calH$ to itself. Given that $\calC$ is a conjugation, we say that $T$ is \emph{$\calC$-symmetric}, or \emph{complex symmetric with respect to $\calC$}, if
	\begin{align*}
	\calC T  \calC = T^*.
	\end{align*}
	If such a $\calC$ exists, we say that $T$ is \emph{complex symmetric}.
\end{defn}

Recall that Proposition~\ref{adj-a0} provides explicitly the adjoint functions $C_{az+b}^*$. The following two results follow easily (cf. \cite{DMK}).

\begin{prop}\label{adj-moreresults}
	The following are true.
\begin{enumerate}
\item $C_{z+b}^* = C_{z+\ol{b}}$.
\item If $a > 1$, then $\ker C_{az+b}^*$ is non-trivial.
\end{enumerate}
\end{prop}

\begin{proof}
(1) follows directly from Proposition~\ref{adj-a0} (note that since $a=1$ we have $m_n=n$). Let us now prove (2). If $\lambda_1\neq 0$, then $\lambda_{m_1}=a\lambda_1 >\lambda_1$, and so $m_1>1$. It follows from \eqref{C-adj} that since $j\mapsto m_j$ is strictly increasing, we have $C_{az+b}^* e_1= 0$, that is $e_1\in \ker C_{az+b}^*$. If $\lambda_1=0$, then $m_1=1$ and arguing as before, we show that $e_2\in \ker C_{az+b}^*$.
\end{proof}

%
%
%

As already mentioned, composition operators associated to non-constant analytic symbols are injective. Thus we immediately get the following. 
\begin{cor}\label{r>1adjoint}
	Suppose $\beta_* \neq \pm\infty$. Let $a>1$ and suppose $C_{az+b}$ is a bounded composition operator on $\Hblam$. Then, no symbol $\psi$ exists such that $C_{az+b}^* = C_\psi$ (here, $\psi$ need not be a polynomial).
\end{cor}

\begin{proof}
First note that since $a>1$, then there cannot exists a constant symbol $\psi$ such that $C_{az+b}^* = C_\psi$. Then it remains to apply Proposition~\ref{adj-moreresults} (2) to get the result.
\end{proof}

\begin{prop}
	Let $\xi(z) = z + ci, c\in\R$. Let $b\in\C$ and suppose that $\Re(b)\geq 0$. Then $C_{z+b}$ is $J_\xi$-symmetric.
\end{prop}

\begin{proof}
Remind that $e_n(z)=e^{-\lambda_n z}$, $n\geq 1$. Since $a=1$, $m_n=n$, $n\geq 1$, and so, on one hand, we have
\[
(C_{z+b} J_\xi e_n)(z)= C_{z+b}\bigbracket{e^{-\lambda_n ci} e^{-\lambda_n z}}=e^{-\lambda_n (b+ci)} e^{-\lambda_{n} z}.
\]
On the other hand, by \eqref{C-adj}, we have
\[
(J_\xi C_{z+b}^* e_n)(z) = J_\xi\bigbracket{e^{-\lam_n \ol{b}} e^{-\lam_n z}}= e^{-\lam_n (b+ci)} e^{-\lam_n z}.
\]
Thus $C_{z+b} J_\xi$ and $J_\xi C_{z+b}^*$ coincide on an orthogonal basis, whence $C_{z+b} J_\xi=J_\xi C_{z+b}^*$. Hence $C_{z+b}$ is $J_\xi$-symmetric.
\end{proof}

The above proves that all bounded composition operators with symbol $C_{z+b}$ is complex symmetric. To determine complex symmetry property of $C_{az+b}$ when $a > 1$, we recall the following result.

\begin{lem}[\cite{garcia2006complex}]\label{lem:Garcia}
	Let $T$ be a complex symmetric operator. Then, $\dim\ker T^* = \dim\ker T$.
\end{lem}

Now we may prove the following result.

\begin{prop}
	Let $\beta_* \neq\pm\infty$ and $a > 1$. Then, $C_{az+b}$ is not complex symmetric on $\Hblam$.
\end{prop}

\begin{proof}
It follows immediately from Lemma~\ref{lem:Garcia} and Proposition~\ref{adj-moreresults} and the fact that $C_{az+b}$ is one to  one.
\end{proof}

Finally we consider the constant case.

\begin{lem}[\cite{GarWog}]\label{rankonecs}
	Any rank one operator is complex symmetric.
\end{lem}

We deduce immediately the following.
\begin{cor}\label{rankoneprop}
	Let $\lam_1 = 0$ and $\beta_* \neq\pm\infty$. Then, $C_b$ is complex symmetric on $\Hblam$.
\end{cor}

We can summarize all the previous results to obtain the following characterization of complex symmetric composition operators.

\begin{thm}\label{catchallcomplexconj}
Let $\beta_* \neq \pm\infty$ and $C_\vphi$ be a bounded composition operator on $\Hblam$. 
Then:
\begin{enumerate}
		\item If $\vphi(z) = z+b$, then $C_\vphi$ is complex symmetric. More precisely, $C_\vphi$ is $J_\xi$-symmetric for all composition conjugations $J_\xi$ where $\xi(z) = z+ci, c\in\R$.
		\item If $\vphi(z) = az+b$ where $a\in R_1(\lam_n)$, then $C_\vphi$ is never complex symmetric.
		\item If $\lambda_1=0$, then $C_b$ is complex symmetric.
	\end{enumerate}

\end{thm}

\section{\bf On similar results when $\beta_* = \infty$}

In \cite{DK20}, a characterization for boundedness of $C_\vphi$ on $\Hblam$ in the case of $\beta_* = \infty$ was proven. Specifically, the authors prove the following analogue to Theorem \ref{cri-bdd}:

\begin{thm}[\cite{DK20}]\label{inftycase}
	Let $\beta_* = \infty$. Let $\vphi$ be an entire function. The following are true.
	\begin{enumerate}
		\item If $\lam_1 > 0$, then $C_\vphi$ is a bounded composition operator on $\Hblam$ if and only if $\vphi(z) = z + b$, $\Re(b) \geq 0$.
		\item If $\lam_1 = 0$, then $C_\vphi$ is a bounded composition operator on $\Hblam$ if and only if $\vphi(z) = z + b$, $\Re(b) \geq 0$, or if $\vphi$ is constant.
	\end{enumerate}
\end{thm}

The authors also estimate $\|C_b\|_{\rm op}$ and compute $\|C_{z+b}\|_{\rm op}$. They conjectured that results on certain properties of bounded composition operators on spaces $\{\HbS\}$, a proper subset of the set of spaces $A := \{\Hblam: \beta_* = \infty\}$, hold for bounded composition operators on all spaces in $A$.

We remark that up to minor modifications in the proof, the following results in our paper hold for all spaces in $A$. \textbf{In the following list, the condition $\beta_* = \infty$ is assumed.} We will use the following observation:
\begin{equation}\label{simplification}
r_n(1,\Re(b)) = e^{-\lam_n \Re(b)}.
\end{equation}

\begin{itemize}
	\item \underline{Operator norms:} Propositions \ref{bdd-a0} and \ref{cri-bdd-a>1}. In the latter case the operator norm simplifies via \eqref{simplification} to $\|C_{z+b}\|_{\rm op} = e^{-\lam_1\Re(b)}$ (cf. \cite{HKZalter}).
	\item \underline{Essential norms and compactness:} Proposition \ref{cpt-a0} and Theorem \ref{ess-norm} (cf. \cite{HHK, HKZalter}).
	\item \underline{Schatten class and adjoint:} Propositions \ref{Sp-a1} (cf. \cite{WY}) and \ref{adj-a0}.
	\item \underline{Compact differences:} Theorem \ref{compactdiffcharacterisation} and Corollary \ref{ehhhhhhhhhhhhhh}. In the former, condition (2)(i) simplifies to $\Re(b) = \Re(b') = 0$ (c.f. \cite{HHK}).
	\item \underline{Closed range:} Section 5.1 preamble (on $C_b$) and Proposition \ref{prop:closeness}. The latter simplifies via \eqref{simplification} to the statement that $R(C_{z+b})$ is closed if and only if $\Re(b) = 0$.
	\item \underline{Cyclicity:} Section 5.2 preamble (on $C_b$) and Proposition \ref{cyclicity-and-supercyclicity}.
	\item \underline{Conjugations:} Theorem \ref{conjugationswoohoo!} (cf. \cite{DMK}). It is worth noting that due to Theorem \ref{inftycase}, we can weaken the assumption on $\xi$ to simply assuming $\xi$ is entire.
	\item \underline{Complex symmetry:} Theorem \ref{catchallcomplexconj}, leaving out statement (2) (cf. \cite{DMK}).
\end{itemize}

\bigskip
%


\begin{thebibliography}{99}
	
\bibitem{Apo} T.M. Apostol, \textit{Modular functions and Dirichlet Series in Number Theory}, Springer-Verlag, New York, 1990.

\bibitem{BQS} F. Bayart, H. Queff\'elec, K. Seip, Approximation numbers of composition operators on Hp spaces of Dirichlet series, {\textit{Ann. Inst. Fourier (Grenoble)}} 66 (2016), no. 2, 551--588. 

\bibitem{Bayart} F. Bayart, Hardy spaces of Dirichlet series and their composition operators, {\textit{Monatsh. Math.}} 136 (2002), no. 3, 203--236.

\bibitem{CM} C. Cowen, B. MacCluer, \textit{Composition operators on spaces of analytic functions}, Studies in Advances in Mathematics, 1995.  
\bibitem{DHKQ} M.L. Doan, B. Hu, L.H. Khoi, H. Queffelec, Approximation numbers for composition operators on spaces of entire functions, \textit{Indag. Math. (N.S.)} 28 (2017), no. 2, 294--305.

\bibitem{DK16} M.L. Doan, L.H. Khoi, Hilbert spaces of entire functions and composition operators, \textit{Compl. Anal. Oper. Theory} 10 (2016), 213--230.

\bibitem{Doan-Khoi} M.L. Doan, L.H. Khoi, Closed range and cyclicity of composition operators on Hilbert space of entire functions, \textit{Complex Var. Elliptic Equ.}, 63 (2018), no. 11, 1558--1569.

\bibitem{DK20} M.L. Doan, L.H. Khoi, Complete characterization of bounded composition operators on the general weighted Hilbert spaces of entire Dirichlet series, \textit{North-West. Eur. J. Math.} 6 (2020), 91--106.

\bibitem{DMK} M.L. Doan, C. Mau, L.H. Khoi, Complex symmetry of composition operators on Hilbert spaces of entire Dirichlet series, \textit{Vietnam J. Math.} 47 (2019), no. 2, 443--460.

\bibitem{garcia2006complex} S. Garcia, M. Putinar, Complex symmetric operators and applications, \textit{Trans. Amer. Math. Soc.} 358 (2006), no. 3, 1285--1315.

\bibitem{GarWog} S. Garcia, W. Wogen, Some new classes of complex symmetric operators, \textit{Trans. Amer. Math. Soc.} 362 (2010), no. 11, 6065--6077.

\bibitem{GH99} J. Gordon, H. Hedenmalm, The composition operators on the space of Dirichlet series with square summable coefficients, \textit{Michigan Math. J.} 46 (1999), 313--329.

\bibitem{Halmos} P. Halmos, \textit{A Hilbert space problem book}, Second edition. Graduate Texts in Mathematics, 19. Encyclopedia of Mathematics and its Applications, 17. Springer-Verlag, New York-Berlin, 1982.

\bibitem{HR15} G.H. Hardy, M. Riesz, \textit{The General Theory of Dirichlet Series}, Stechert-Hafner, Inc., New York, 1964.
	
\bibitem{Hed04} H. Hedenmalm, Dirichlet series and Functional Analysis. {\it In: The Legacy of Niels Henrik Abel}, The Abel Bicentennial, Oslo 2002 (O. A. Laudal, R. Piene, editors), Springer-Verlag, 2004, 673--684.

\bibitem{HLS} H. Hedenmalm, P. Lindqvst, K. Seip, A Hilbert space of Dirichlet series and  systems of dilated functions in $L^{2}(0,1)$, \textit{Duke Math. J.} 86 (1997), 1--37.

\bibitem{Hilden} H. Hilden and L. Wallen, Some cyclic and non-cyclic vectors of certain operators, \textit{Indiana Univ. Math. J.} 23 (1973/74), 557--565.

\bibitem{HK} X. Hou, L.H. Khoi, Some properties of composition operators on entire Dirichlet series with real frequencies, \textit{Compt. Rend. Math. Acad. Sci. Paris}, Ser. I 350 (2012), no. 3-4, 149--152.

\bibitem{HHK} X. Hou, B. Hu, L.H. Khoi, Hilbert spaces of entire Dirichlet series and composition operators, \textit{J. Math. Anal. Appl.} 401 (2013), 416--429.

\bibitem{HKZalter} B. Hu, L.H. Khoi, R. Zhao, Topological structure of the spaces of composition operators on Hilbert spaces of Dirichlet series, \textit{Z. Anal. Anwend. (J. Anal. Appl.)} 35 (2016), 267--284.

\bibitem{HKZ} B. Hu, L.H. Khoi, K. Zhu, Frames and operators in Schatten classes, \textit{Houston J. Math.} 41 (2015), no. 4, 1191--1219.
 
\bibitem{QS} H. Queff\'elec H., K. Seip,  Approximation numbers of composition operators on the $H^2$ space of Dirichlet series, {\textit{J. Funct. Anal.}} 268 (2015), no.6, 1612--1648.

\bibitem{Q} H. Queff\'elec H., Espaces de séries de Dirichlet et leurs opérateurs de composition,  {\textit{Ann. Math. Blaise Pascal}}, 22 (2015), no. S2, 267--344. 
 
\bibitem{Seubert} S. Seubert, Cyclic vectors on shift coinvariant subspaces, \textit{Rocky Mountain J. Math.} 24 (1994), no. 2, 719--727.


\bibitem{Sh} J. Shapiro, \textit{Composition operators and classical function theory}, Springer-Verlag, New York, 1993.

\bibitem{Valiron} G. Valiron, \textit{Th\'eorie g\'en\'erale des s\'eries de Dirichlet}, M\'emorial des Sciences Math\'ematiques, Fascicule 17 (1926).

\bibitem{WY} M. Wang, X. Yao, Some properties of composition operators on Hilbert
spaces of Dirichlet series, \textit{Complex Var. Elliptic Equ.} 60 (2015), no. 7, 992--1004.

\bibitem{YDT} J.R. Yu, X.Q. Ding, F.J. Tian, \textit{On the distribution of values of Dirichlet series and random Dirichlet series}, Press in Wuhan
Univ., Wuhan, China, 2004.

\end{thebibliography}
\end{document}